\documentclass{siamart220329}
\usepackage{amsmath}
\usepackage{amssymb}
\usepackage{graphicx}
\usepackage{caption}
\usepackage{xcolor}
\usepackage{hyperref}
\usepackage{bm}
\usepackage{cleveref}
\usepackage{dutchcal}
\usepackage{algorithm}

\ifpdf
\hypersetup{
  pdftitle={Design and control of quasiperiodic patterns of particles with standing acoustic waves},
  pdfauthor={E. Cherkaev, F. Guevara Vasquez, and C. Mauck}
}
\fi

\newcommand{\vF}{\bm{F}}

\newcommand{\vx}{\bm{x}}
\newcommand{\vc}{\bm{c}}
\newcommand{\vk}{\bm{k}}
\newcommand{\ve}{\bm{e}}
\newcommand{\vy}{\bm{y}}
\newcommand{\vv}{\bm{v}}
\newcommand{\vz}{\bm{z}}
\newcommand{\vu}{\bm{u}}
\newcommand{\vh}{\bm{h}}
\newcommand{\vw}{\bm{w}}
\newcommand{\va}{\bm{a}}

\newcommand{\vs}{\bm{s}}
\newcommand{\vn}{\bm{n}}

\newcommand{\vb}{\bm{b}}

\newcommand{\vq}{\bm{q}}
\newcommand{\vt}{\bm{t}}
\newcommand{\vzero}{\bm{0}}

\newcommand{\veps}{\bm{\varepsilon}}

\newcommand{\vgamma}{\bm{\gamma}}
\newcommand{\vlambda}{\bm{\lambda}}

\newcommand{\mA}{\mathsf{A}}
\newcommand{\mM}{\mathsf{M}}
\newcommand{\mD}{\mathsf{D}}
\newcommand{\mQ}{\mathsf{Q}}
\newcommand{\mK}{\mathsf{K}}
\newcommand{\mI}{\mathsf{I}}
\newcommand{\mP}{\mathsf{P}}
\newcommand{\mR}{\mathsf{R}}

\newcommand{\mZ}{\mathsf{Z}}

\renewcommand{\Re}{\mathrm{Re}}
\renewcommand{\Im}{\mathrm{Im}}
\newcommand{\real}{\mathbb{R}}
\newcommand{\complex}{\mathbb{C}}

\newcommand{\znat}{\mathbb{Z}}

\newsiamremark{remark}{Remark}
\newsiamremark{example}{Example}

\DeclareMathOperator{\linspan}{span}
\DeclareMathOperator{\diag}{diag}
\DeclareMathOperator{\range}{range}
\DeclareMathOperator{\rank}{rank}
\DeclareMathOperator{\argmin}{argmin}
\DeclareMathOperator{\nullspace}{null}

\newcommand{\cA}{\mathcal{A}}

\title{Design and control of quasiperiodic patterns of particles with standing acoustic waves
\thanks{Submitted to the editors DATE.
\funding{The authors acknowledge support from the National Science Foundation grants DMS-2008610, 
DMS-2136198, 
DMS-2111117,  
DMS-2206171,  
DMS-1715680   
and the Army Research Office under Contract No. W911NF-16-1-0457.
}}}

\author{Elena Cherkaev\thanks{
	Mathematics Department, University of Utah, Salt Lake City, UT 84112
	(\email{elena@math.utah.edu}, \email{fguevara@math.utah.edu}).
} 
\and Fernando Guevara Vasquez\footnotemark[2]
\and China Mauck\footnotemark[2]
\thanks{Currently: STV Incorporated, 200 W Monroe St {\#}1650, Chicago, IL 60606.}
}

\begin{document}

\maketitle

\begin{abstract}
We develop a method to design tunable quasiperiodic structures of particles
 suspended in a fluid by controlling standing acoustic waves. One application of
 our results is to ultrasound directed self-assembly, which allows fabricating
 composite materials with desired microstructures. Our approach is based on
 identifying the minima of a functional, termed the acoustic radiation
 potential, determining the locations of the particle clusters.  
This functional can be viewed as a two- or three-dimensional slice of a similar
functional in higher dimensions as in the cut-and-project method of constructing
quasiperiodic patterns. The higher dimensional representation allows for
translations, rotations, and reflections of the patterns. Constrained
optimization theory is used to characterize the quasiperiodic designs based on
local minima of the acoustic radiation potential and to understand how changes
to the controls affect particle patterns. We also show how to transition
smoothly between different controls, producing smooth transformations of the
quasiperiodic patterns.   
The developed approach unlocks a route to creating  
tunable quasiperiodic and moir{\'e} structures known for their unconventional
superconductivity and other extraordinary properties.   
Several examples of constructing quasiperiodic structures,  
including in two and three dimensions, are given. 
\end{abstract}



\headers{Design and control of quasiperiodic patterns}{E. Cherkaev, F. Guevara Vasquez, and C. Mauck}

\section{Introduction}
\label{sec:intro}

Small (compared to the wavelength) particles immersed in a fluid move under the pressure of acoustic waves and can be forced to create specific targeted structures \cite{Greenhall:2015:UDS,Guevara:2019:PPA, Prisbrey:2017:UDS}. We consider the problem of designing tunable quasiperiodic structures by controlling the acoustic wavefield. 
Quasiperiodic arrangements of particles within a fluid can be obtained by
establishing particular patterns of standing acoustic waves using ultrasound transducers
\cite{Cherkaev:2021:WDA}. 
The current work develops methods for predicting the locations of the
particles' clusters and controlling 
continuous changes of 
the quasiperiodic
structures, including their translations, rotations, and reflections. The approach can also be used to create moir{\'e} structures known for their unusual qualities, such as superconductivity and other
extraordinary properties, see e.g. \cite{Cao:2018:USM,Kamiya:2018:DSQ,Morison:2022:CP}. 

We relate constructed quasiperiodic structures to aperiodic tilings of the
plane (or space). Perhaps the best-known aperiodic tiling of the plane was designed by
Penrose \cite{Gardner:1977:ENT, Penrose:1974:ROP, Penrose:1979:PCN}. De Bruijn \cite{deBruijn:1981:ATP} showed that every Penrose tiling can be obtained by projecting a $5-$dimensional periodic lattice
to a $2-$dimensional plane. This ``cut-and-project" method is not restricted to constructing Penrose
tilings and applies to other kinds of quasiperiodic tilings in other dimensions \cite{Austin:2005:PTT, Gahler:1986:EGG, Korepin:1988:QPT,Wellander:2018:MMAS}. 

In addition to theoretical studies of quasiperiodic tilings \cite{Gardner:1977:ENT, Penrose:1974:ROP, Penrose:1979:PCN,deBruijn:1981:ATP}, physical
quasicrystals
\cite{Shechtman:1984:MPL,Yamamoto:1996:CQC,Janot:2012:Q,Stadnik:1999:PPQ} have
been realized and studied beginning with the 
discovery by Dan Shechtman of
an Al-Mn alloy with icosahedral symmetry \cite{Shechtman:1984:MPL}. Since this
discovery, many groups have synthesized
quasicrystals in various ways. In particular, quasiperiodic structures have
been observed experimentally with lasers
\cite{Roichman:2005:HAQ,Wang:2006:ROT,Mikhael:2008:ALT,Sun:2015:FTF} and
ultrasound waves \cite{Espinosa:1993:AQC}.  Materials with quasicrystalline structures have  been shown to possess  remarkable, extraordinary physical properties (see, e.g., \cite{Janot:2012:Q,Stadnik:1999:PPQ}).


Methods of creating quasiperiodic structures developed in the present paper apply to any linear wave-like phenomena that can be modeled with
finitely many plane waves. A potential application of our work is ultrasound-directed self-assembly \cite{Greenhall:2015:UDS, Prisbrey:2017:UDS}, a material fabrication method in which a liquid resin containing small particles is placed in a reservoir lined with ultrasound transducers. 
The transducers operating at
a fixed frequency generate a standing acoustic wave in the liquid resin 
forcing the particles to cluster at certain locations as 
they are subject to an acoustic radiation force \cite{Gorkov:1962:FAS,
King:1934:ARP, Bruus:2012:ARF, Settnes:2012:FAS}; this force can be described
in terms of an acoustic radiation potential (ARP).  
Curing the resin results
in a material with particles 
aggregated 
at locations determined by the standing ultrasound
wave. 
Other situations where particles are arranged with ultrasound include
microfluidics \cite{Friend:2011:MAM, Wiklund:2013:UEI} and acoustic tweezers
\cite{Meng:2019:AT, Courtney:2014:ITM, Marzo:2019:HAT, Ozcelik:2018:ATL}. Further potential  application of the developed method is to assemble structures with light, see e.g. \cite{Ashkin:1987:OTM,Neuman:2004:OT,Nieminen:2007:POT}.   
Moreover, the created patterns of the quasiperiodic structures could change with time, 
resulting in time-dependent metamaterials, as in e.g. acoustically configurable crystals \cite{Caleap:2014:ATC, Silva:2019:PPU}.

\subsection{Mathematical model}
\label{sec:mathmod}
The pressure field $p$ generated by a standing acoustic wave satisfies the Helmholtz
equation $\Delta p + k^2 p = 0$, where $k = 2\pi/\lambda$ is the wavenumber, and
$\lambda$ is the wavelength (see e.g. \cite{Colton:1998:IAE}.) When small
(relative to $\lambda$) particles are placed in the standing acoustic wave
field, they are subject to an acoustic radiation force \cite{Gorkov:1962:FAS,
King:1934:ARP, Bruus:2012:ARF, Settnes:2012:FAS}. This force can be described
in terms of an acoustic radiation potential (ARP) $\psi$ as $\vF = -\nabla \psi$.
Since the net forces point in the direction of the negative gradient of the
potential $\psi$, microparticles subject to these forces cluster at the local
minima of $\psi$. Thus, the ARP (and specifically its local minima) predicts
locations of the particle clusters.

In general, given a solution $p(\vx)$ to the Helmholtz equation in dimension $d$, we can define a quadratic potential of the form
\begin{equation}
 \psi(\vx) = 
 \begin{bmatrix}
 p(\vx) \\ \nabla p(\vx)
 \end{bmatrix}^*
 \mA
 \begin{bmatrix}
 p(\vx) \\
 \nabla p(\vx)
 \end{bmatrix},
\label{eq:general-arp}
\end{equation}
where $\mA$ is a $(d+1)\times(d+1)$ Hermitian matrix. In particular, the physical ARP in dimensions $d=2$ and $3$ comes from using \eqref{eq:general-arp} and choosing
\begin{equation*}
 \mA = \begin{bmatrix}
 \mathfrak{a} & \\
 & -\mathfrak{b}\mI_d
 \end{bmatrix},
\end{equation*}
where $\mI_d$ is the $d\times d$ identity matrix and $\mathfrak{a}, \mathfrak{b}$ are non-dimensional constants that depend on the densities and compressibilities of the fluid and particles. Thus when $d=2$ or 3, we get the expression for the ARP used in, e.g., \cite{Bruus:2012:ARF,Settnes:2012:FAS,Gorkov:1962:FAS}: 
\begin{equation}
 \psi(\vx) = \mathfrak{a}|p(\vx)|^2 - \mathfrak{b}\|\nabla p(\vx)\|^2,
 \label{eq:ARP1}
\end{equation}
where the norm is the Euclidean norm. We now consider solutions to the Helmholtz equation in dimension $d$ that are 
finite linear combinations of plane waves, i.e.
\begin{equation}
  p(\vx;\vu) = \sum_{j=1}^N \alpha_j \exp[i \vk_j \cdot \vx ] + \beta_j
  \exp[-i \vk_j \cdot \vx], \quad \vx\in\real^d,
 \label{eq:planewaves}
\end{equation} where the wavevectors $\vk_j \in \real^d$ have the same length $k$ and $\vu=[\alpha_1, \ldots, \alpha_N$, $\beta_1, \ldots,
\beta_N]^T$ is a vector 
containing $2N$ complex amplitudes. When $d=2$ or $3$,
these fields can be obtained experimentally with $N$ parallel pairs of
transducers (illustrated in \cref{fig:transducers} for the case $d=2$ and
$N=3$). In that case, the complex amplitudes $\alpha_j$ and $\beta_j$ describe the
amplitudes and phases of the voltages driving each transducer.\footnote{This is
a simplification, since it ignores any electric and acoustic impedances. A more
accurate modeling of the transducer controls is outside of the scope of this
work.} If 
the number of pairs of transducers coincides with the space dimension 
($N=d$),  the field \eqref{eq:planewaves} is periodic, and thus the particles
form periodic patterns within a particular subset of the achievable
crystallographic symmetries in dimension $d$ \cite{Guevara:2019:PPA}. The case $N>d$ 
has been shown theoretically and experimentally to
produce quasiperiodic arrangements \cite{Cherkaev:2021:WDA}. This is done by 
noticing that the ARP in the dimension $d$ is the restriction to
dimension $d$ of a periodic ARP-like potential quantity in dimension $N$. 
This is similar to creating quasiperiodic patterns 
by projecting a structure periodic in higher-dimensional space onto a lower-dimensional subspace in the cut-and-project method. 
Here we exploit
this higher-dimensional point of view to devise explicit ways in which the
quasiperiodic patterns created by the nanoparticles in a fluid can be controlled 
and tuned 
through, e.g., translations, rotations, and reflections.

\begin{figure}
\begin{center}
 \includegraphics[width=0.3\textwidth]{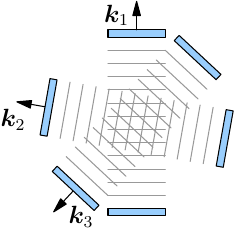}
\end{center}
 \caption{Physical setup to obtain standing plane waves of the form
 \eqref{eq:planewaves}. Here, we give an example in two dimensions with 3
 parallel pairs of linear transducers (in blue). The region where the fields in this
 setup agree with \eqref{eq:planewaves} depends on the wavelength and
 the distance to the transducers.}
 \label{fig:transducers}
\end{figure}

\subsection{Contents}
The presentation in the paper is arranged as follows. We start in \cref{sec:quasiperiodic} by describing a  quasiperiodic acoustic radiation potential (ARP). We use a method similar to the cut-and-project method
\cite{deBruijn:1981:ATP}, relying on a higher-dimensional ARP-like potential. 
In \cref{lem:condition}, we 
give a novel definition of quasiperiodicity (initially proposed in \cite{Cherkaev:2021:WDA}) that includes
necessary and sufficient conditions ensuring that the wavefield is quasiperiodic.
We show quasiperiodicity of the ARP generated by a superposition of 
$N$ plane waves in the $d$-dimensional space with $N>d$, 
and demonstrate 
examples of the associated aperiodic tilings of the plane for some
well-known rotational symmetries. 
Then in \cref{sec:derivatives}, we derive
explicit formulas for the gradient and Hessian of the ARP-like quantity. These
formulas allow us to independently confirm results in \cite{Guevara:2019:PPA}
about the level sets and minima of the ARP. However, these minima do
not necessarily reflect the positions of the particle clusters. To find the
actual positions of the particles, we calculate the gradient and Hessian of the ARP-like
quantity restricted to a lower-dimensional space (\cref{sec:in-subspace}). We
also demonstrate how to 
continuously transform the structures to 
translate, rotate, and reflect the
created particle patterns. In \cref{sec:constrained}, we formulate the
problem of finding the minima of the ARP in $d-$dimensions as a constrained optimization problem in $N-$dimensions. This allows us to relate certain
infinitesimally small phase changes in the transducers to changes in the ARP.
A strategy for smoothly transitioning between two ARP patterns while keeping a
constant power is given in \cref{sec:power}. In \cref{sec:num} we present an extension of   
the theory to 3D and to constructing time dependent structures and illustrate 3D ARP arrangements.
We conclude with a summary and future work in \cref{sec:summary}.

\section{Quasiperiodicity of the ARP}
\label{sec:quasiperiodic}

First, we recall the result in \cite{Cherkaev:2021:WDA} that shows that when
using $N>d$ parallel pairs of transducers in $d$ dimensions, the resulting
superposition of plane waves and its associated ARP can be quasiperiodic. The
main idea is to show that these $d-$dimensional quantities are the restriction
to a $d-$dimensional subspace of an $N-$dimensional \textit{periodic} function. 
We recall that a
function $f$ defined in $\real^d$  is \textit{quasiperiodic} if it is not
periodic and yet we can find a periodic function $g$ in $\real^N$ with $N>d$ and
a matrix $\mK \in \real^{d \times N}$ such that $f(\vx) = g(\mK^T \vx)$ (see e.g. \cite{Yamamoto:1996:CQC}).
The
$N-$dimensional representation allows us to relate the constructed particle patterns to quasiperiodic tilings in $d$ dimensions (e.g., Penrose tilings
\cite{Gardner:1977:ENT,Penrose:1974:ROP,Penrose:1979:PCN}). 
 
\subsection{Quasiperiodicity of a superposition of plane waves}
\label{sec:pqp}
We observe that a superposition of plane waves of the form
\eqref{eq:planewaves} can always be expressed as the restriction to a
$d-$dimensional subspace of a particular periodic superposition of plane waves in
$N$ dimensions. To be more precise, let $p_d(\vx;\vu)$ be a superposition of $N$ plane waves which is defined for $\vx \in \real^d$ as in \eqref{eq:planewaves} with wave vectors $\vk_1,
\dots, \vk_N \in \real^d$ and $N\geq d$. Similarly, let $p_N(\vy;\vu)$ be the superposition of $N$ plane waves defined for $\vy \in \real^N$ by
\begin{equation}
 p_N(\vy;\vu) = \sum_{j=1}^N \alpha_j \exp[i\ve_j \cdot \vy] + \beta_j \exp[-i\ve_j \cdot \vy],
 \label{eq:pn}
\end{equation}
where the wavevectors are now $N-$dimensional and given by the canonical basis vectors $\ve_1,\ldots,\ve_N$ of $\real^N$. If we define the matrix
\begin{equation}
    \mK = [\vk_1,\ldots,\vk_N] \in \real^{d \times N},
\end{equation}
we can see that
\begin{equation}
 p_d(\vx;\vu) = p_N(\mK^T\vx;\vu).
 \label{eq:rest}
\end{equation}
Indeed, the equality \eqref{eq:rest} comes from realizing that $\ve_j \cdot (\mK^T \vx) = \vk_j \cdot \vx$, $j=1,\ldots,N$. The equality \eqref{eq:rest} means that $p_d(\vx;\vu)$ is the restriction of the $N-$dimensional superposition of plane waves $p_N(\vy;\vu)$ to the affine subspace $\{\vy = \mK^T\vx ~|~ \vx \in \real^d\}$.

By the same observation, the restriction of \eqref{eq:pn} to the affine subspace
$\{ \vy = \mK^T \vx+\vgamma ~|~ \vx \in \real^d \}$, where $\vgamma \in \real^N$ is fixed, is a superposition of plane waves in $\real^d$, namely
\begin{equation}
 p_N(\mK^T\vx + \vgamma;\vu) = p_N(\mK^T\vx;\exp[i\mD(\vgamma)]\vu)  = p_d(\vx;\exp[i\mD(\vgamma)]\vu),
 \label{eq:shifts}
\end{equation} 
where $\mD(\vgamma) = \diag(\vgamma,-\vgamma)$ and $\diag([a_1,\dots,a_n])$ is the matrix with diagonal elements $a_1,\dots,a_n$. Therefore, shifting the affine subspace by $\vgamma$ is equivalent to applying a phase shift to the transducer parameters $\vu$. For simplicity, we choose $\vgamma = \vzero$, whenever possible. An application of \eqref{eq:shifts} is shown in \cref{sec:qp:ca} to achieve spatial translations of the ARP.

Clearly $p_N(\vy;\vu)$ in \eqref{eq:pn} is a periodic function in $\real^N$, with period being the hypercube $[0,2\pi]^N$.  Since by \eqref{eq:rest} the superposition $p_d(\vx;\vu)$ is the restriction of a periodic function in dimension $N$, we can give a condition on $\mK$ that guarantees quasiperiodicity of $p_d(\vx;\vu)$.  
It is sufficient to ensure that $p_d(\vx;\vu)$ is periodic along at most $d-1$ linearly independent directions, i.e., there is at least one direction along which $p_d(\vx;\vu)$ is not periodic. 
The next lemma provides a necessary and sufficient condition on $\mK$, ensuring $p_d(\vx;\vu)$ is quasiperiodic, which was first introduced in \cite{Cherkaev:2021:WDA}.

\begin{lemma}
\label{lem:condition}
Assume $\rank(\mK)=d$ and that $p_N(\vy;\vu)$ is periodic with primitive cell $[0,2\pi]^N$. The function $p_d(\vx;\vu) = p_N(\mK^T\vx;\vu)$ is quasiperiodic if and only if
$$\dim(\range(\mK^T) \cap (2\pi)\znat^N) \le d-1,$$
where the dimension of a set is defined as the dimension of its linear span.
\end{lemma}
\begin{proof}
We will prove both directions by contraposition, i.e., we prove that $p_d(\vx;\vu)$ is periodic if and only if $\dim(\range(\mK^T) \cap (2\pi)\znat^N)=d$.

First, assume that $\dim(\range(\mK^T) \cap (2\pi)\znat^N) = d$. Then we can find a set of $d$ linearly independent vectors $\{\vb_1, \ldots, \vb_d\} \subset \range(\mK^T) \cap (2\pi)\znat^N$. Because $\rank(\mK) =d$, there are $d$ linearly independent vectors 
$\{\va_1,\ldots,\va_d\} \subset \real^d$ such that $\vb_j = \mK^T\va_j$, $j=1,\ldots, d$. Then for all $n_1, \ldots, n_d \in \znat$,
\begin{align*}
p_d(\vx + n_1\va_1 + \cdots + n_d\va_d;\vu) &= p_N(\mK^T\vx + n_1\mK^T\va_1 + \cdots + n_d\mK^T\va_d;\vu) \\
&= p_N(\mK^T\vx + n_1\vb_1 + \cdots + n_d\vb_d;\vu) \\
&= p_N(\mK^T\vx;\vu) \\
&= p_d(\vx;\vu),
\end{align*}
where we have used the fact that $p_N(\vy;\vu)$ is periodic on the $(2\pi)\znat^N$ lattice with primitive cell $[0,2\pi]^N$. Therefore $p_d(\vx;\vu)$ is periodic.

To prove the other direction, assume $p_d(\vx;\vu)$ is periodic. Then there exists a set of $d$ linearly independent vectors $\{\va_1,\ldots,\va_d\}\subset \real^d$ such that for all $n_1,\ldots, n_d\in \znat$,
$$p_d(\vx;\vu) = p_d(\vx + n_1\va_1 + \cdots + n_d\va_d;\vu).$$
Define the vectors $\vb_j = \mK^T\va_j$, $j = 1,\ldots,d$. The set of vectors $\{\vb_1,\ldots,\vb_d\}$ is linearly independent because $\mK^T$ has full column rank. For the sake of contradiction, assume there exists a $j' \in \{1,\ldots,d\}$ such that $\vb_{j'} \not\in (2\pi)\znat^N$. Using our assumption of periodicity of $p_d(\vx;\vu)$, we have that for any $n\in\znat$,
\begin{align*}
p_N(\mK^T\vx+n\vb_{j'};\vu) &= p_N(\mK^T\vx + n\mK^T\va_{j'};\vu) \\
&= p_d(\vx+n\va_{j'};\vu) \\
&= p_d(\vx;\vu) \\
&= p_N(\mK^T\vx;\vu).
\end{align*}
Because $\vb_{j'} \not \in (2\pi)\znat^N$, this contradicts the fact that $p_N(\vy;\vu)$ is periodic with primitive cell $[0,2\pi]^N$. Therefore $\dim(\range(\mK^T)\cap(2\pi)\znat^N) = d$.
\end{proof}

In the case $N=2$ and $d=1$ (illustrated in \cref{fig:1D}), the condition of \cref{lem:condition} can be interpreted as follows. Consider the restriction to a line passing through the origin of a function on $\real^2$ that is periodic with unit cell $[0,2\pi]^2$. The restriction is quasiperiodic if and only if the slope of the line is irrational. An example of this construction is illustrated in \cref{fig:tiling-1D}.

\subsection{Quasiperiodicity of the ARP for a superposition of plane waves}
The quadratic potential of the form \eqref{eq:general-arp}
associated with the superposition of plane waves can be written as \cite{Guevara:2019:PPA,Cherkaev:2021:WDA}
\begin{equation}
\psi(\vx;\vu) = \vu^*\mQ(\vx)\vu,
\label{eq:ARP2}
\end{equation}
where 
\begin{equation}
\mQ(\vx) = \mM(\vx)^*\mA\mM(\vx),
\label{eq:Qmat}
\end{equation}
and the $(N+1) \times 2N$ matrix $\mM(\vx)$ is given by
\begin{equation*}
\mM(\vx) = [\mM_+(\vx) \quad \mM_-(\vx)]
\end{equation*}
and
\begin{equation*}
\mM_{\pm}(\vx) =
\begin{bmatrix}
\exp[\pm i\vx^T\mK] \\
\pm i \mK\diag(\exp[\pm i\vx^T\mK])
\end{bmatrix}.
\end{equation*}

Recall that we use the notation $p_N(\vy;\vu)$, $\vy\in\real^N$ and $p_d(\vx;\vu)$, $\vx\in \real^d$ when referring to pressure fields in $N$ dimensions and to pressure fields that have been restricted to affine subspaces $\{\vy = \mK^T\vx ~|~ \vx \in \real^d\}$, and that the two pressure fields are related by $p_d(\vx;\vu) = p_N(\mK^T\vx;\vu)$. We will use a similar notation for ARPs constructed from the fields $p_N(\vy;\vu)$ in $N$ dimensions and ARPs that have been restricted to affine subspaces. Let 
\begin{equation}
 \psi_N(\vy;\vu) = 
 \begin{bmatrix}
 p_N(\vy;\vu) \\ \nabla p_N(\vy;\vu)
 \end{bmatrix}^*
 \mA_N
 \begin{bmatrix}
 p_N(\vy;\vu) \\
 \nabla p_N(\vy;\vu)
 \end{bmatrix}
\label{eq:ARP-N}
\end{equation}

and

\begin{equation}
 \psi_d(\vx;\vu) = 
 \begin{bmatrix}
 p_d(\vx;\vu) \\ \nabla p_d(\vx;\vu)
 \end{bmatrix}^*
 \mA_d
 \begin{bmatrix}
 p_d(\vx;\vu) \\
 \nabla p_d(\vx;\vu)
 \end{bmatrix},
\label{eq:ARP-d}
\end{equation}

where 
\[
\mA_d = \begin{bmatrix}
\mathfrak{a} & \\
& -\mathfrak{b}\mI_d
\end{bmatrix} 
~\text{and}~
\mA_N = \begin{bmatrix}
\mathfrak{a}  & \\
& -\mathfrak{b}\mK^T\mK
\end{bmatrix}.
\]

By taking the gradient of \eqref{eq:rest} and using the chain rule we notice that
\[
 \nabla p_d(\vx;\vu) = \mK \nabla p_N(\mK^T \vx;\vu).
\]
Using this expression in \eqref{eq:ARP-d} we observe that \eqref{eq:ARP-d} and \eqref{eq:ARP-N} are related by \cite{Cherkaev:2021:WDA}
\[ 
 \psi_d(\vx;\vu) = \psi_N(\mK^T\vx;\vu). 
\]
Thus $\psi_d(\vx;\vu)$ is quasiperiodic provided the condition in \cref{lem:condition} for quasiperiodicity of the pressure field $p_d(\vx;\vu)$ is satisfied.
Examples of quasiperiodic ARPs constructed with $N=4$, $N=5$, and $N=6$ transducer directions are given in \cref{ex:arp}.
\begin{example}
\label{ex:arp}
We compute three examples of quasiperiodic ARPs according to \eqref{eq:ARP1}, where the pressure field is computed using \eqref{eq:pn} and \eqref{eq:rest}. We fix $d=2$ to simplify visualization and construct examples with $N=4, 5$, and $6$. We choose the 2D wavevectors to be $\pi/N$ apart in angle ($\vk_j = [\cos((j-1)\pi/N), \sin((j-1)\pi/N)]^T, j=1, \ldots, N$) and choose the vector of transducer parameters $\vu$ such that $\alpha_j = 1$ and $\beta_j = -1$ for $j=1, \ldots, N$. We also fix the ARP parameters $\mathfrak{a} = \mathfrak{b} = 1$. The computed ARPs are shown in \cref{fig:arp}. Note that depending on the transducer parameters $\vu$, quasiperiodic ARPs with $2N-$fold rotational symmetry can be achieved \cite{Cherkaev:2021:WDA}.

\begin{figure}
	\begin{tabular}{ccc}
	\includegraphics[width=0.30\textwidth]{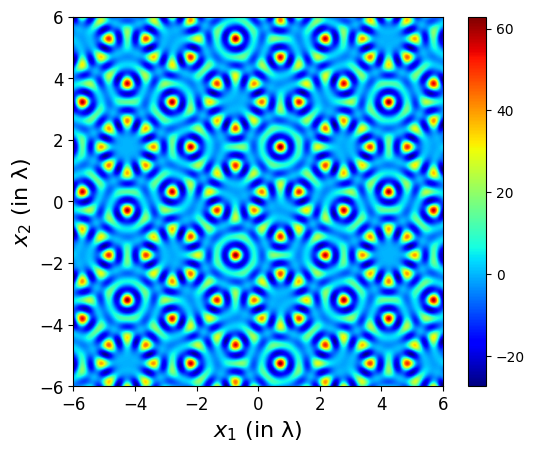} &
	\includegraphics[width=0.30\textwidth]{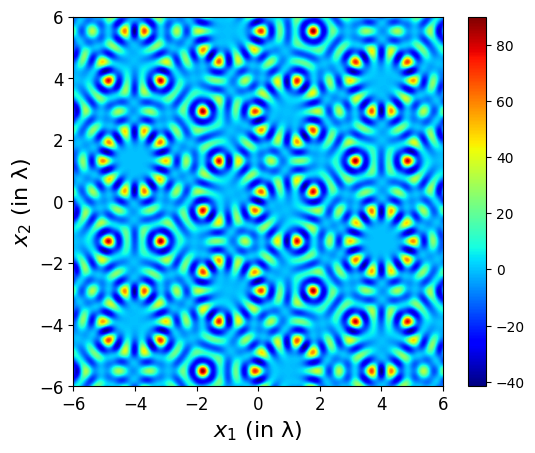} &
	\includegraphics[width=0.30\textwidth]{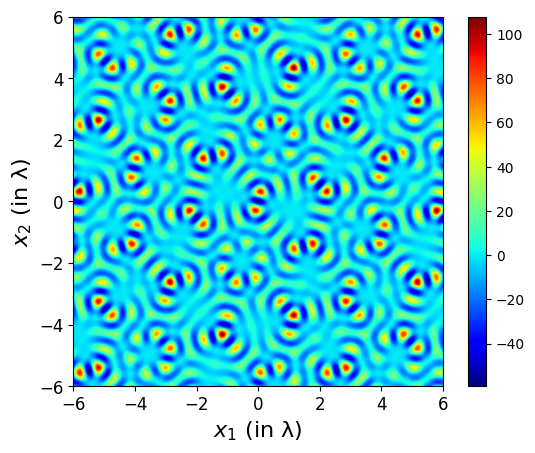} \\
	(a) & (b) & (c)
	\end{tabular}
\caption{Examples of ARPs constructed with (a) $N=4$, (b) $N=5$, and (c) $N=6$ transducer directions. The wavevectors are $\vk_j = [\cos((j-1)\pi/N), \sin((j-1)\pi/N)]^T, j=1, \ldots, N$, the transducer parameters are $\alpha_j=1, \beta_j = -1$, $j=1,\ldots, N$, and the ARP parameters are $\mathfrak{a}=\mathfrak{b}=1$.}
\label{fig:arp}
\end{figure}

\end{example}

\begin{example}
\label{ex:moire}
A linear combination of plane waves of the form \eqref{eq:planewaves} with wavevectors $\vk_1 = [\sqrt{3}/2,1/2]^T$,  $\vk_2 = [-\sqrt{3}/2,1/2]^T$ gives a periodic ARP with an hexagonal Bravais lattice with unit cell spanned by $\va_1 = (2\lambda/\sqrt{3}) [1/2,\sqrt{3}/2]$ and $\va_2 = (2\lambda/\sqrt{3}) [-1/2,\sqrt{3}/2]$ (see e.g. \cite{Kittel:2005:ISP,Guevara:2019:PPA}). If we add the plane waves corresponding to wavevectors $\vk_3=\mR_\theta \vk_1$ and $\vk_4=\mR_\theta \vk_2$, where $\mR_\theta$ is the counterclockwise rotation by an angle $\theta$, then we obtain periodic or quasiperiodic ARPs depending on the rotation angle $\theta$. The angles $\theta$ at which a hexagonal Bravais lattice and its rotation coincide have been characterized since this occurs in twisted bilayer graphene \cite{Shallcross:2008:QIT,LopesdosSantos:2012:CMT}.  The rotation angles $0<\theta<\pi/3$ that give rise to periodicity are given by 
\begin{equation}
 \cos \theta(m,r) = \frac{3m^2 + 3mr+r^2/2}{3m^2 + 3mr+r^2},
 \label{eq:cosine}
\end{equation}
where $m,r$ are non-zero coprime integers. When periodicity occurs, the superlattice is also a hexagonal Bravais lattice with a larger unit cell spanned by vectors $\vt_1$ and $\vt_2$ that can be given as an integer linear combination of either the original lattice vectors $\va_1,\va_2$ or their rotated versions $\mR_{\theta(m,r)} \va_1, \mR_{\theta(m,r)} \va_2$. To be more precise, there are invertible matrices $\mM, \mM' \in \znat^{2\times 2}$ that depend on $m,r$ such that
\[
 [\vt_1, \vt_2 ] = \mM [\va_1, \va_2] = \mM' \mR_{\theta(m,r)} [\va_1, \va_2].
\]
The actual expressions of $\mM$ and $\mM'$ are irrelevant to this example,  but they are included in \cite{LopesdosSantos:2012:CMT}. Periodicity can be verified with the condition of \cref{lem:condition} since:
\[
 \mK^T [\vt_1, \vt_2 ] = \begin{bmatrix}
 [\vk_1,\vk_2]^T\\
 [\vk_3,\vk_4]^T
 \end{bmatrix}
 [\vt_1, \vt_2 ] = 2\pi \begin{bmatrix} \mM\\ \mM' \end{bmatrix},
\]
we see that the $\range(\mK^T)$ contains $d=2$ linearly independent vectors $\mK^T \vt_1,\mK^T \vt_2\in (2\pi)\znat^4$. 

If $\theta$ is not one of the angles 
in \cref{eq:cosine}
then the two hexagonal Bravais lattices do not coincide aside from the origin, and we get a quasiperiodic ARP. Since the right hand side of \eqref{eq:cosine} is rational, any rotation angle with irrational cosine results in a quasiperiodic pattern. One example is $\theta = \pi/6$ as observed in \cite{Yao:2018:QTB}. We illustrate moir\'e patterns in \cref{fig:moire}. 
The case $\theta = \pi/6$ shown in the bottom-right image in \cref{fig:moire} is related to the ARP in \cref{ex:arp}  
with $N=6$ shown in \cref{fig:arp}(c), but without two of the wavenumbers in the sum (\ref{eq:pn}). 

\begin{figure}
 \begin{tabular}{ccc}
	$m = 1, r = 1$ & $m=2, r =1$ & $m=3, r=1$\\
	\includegraphics[width=0.30\textwidth]{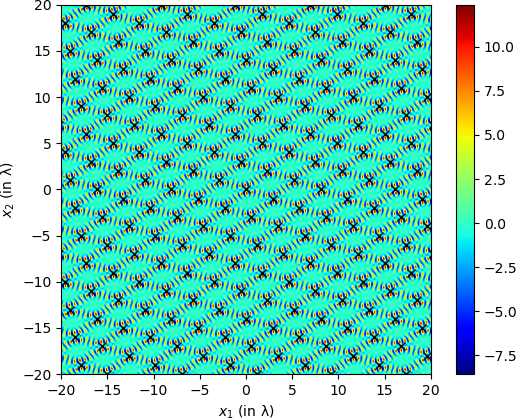} &
	\includegraphics[width=0.30\textwidth]{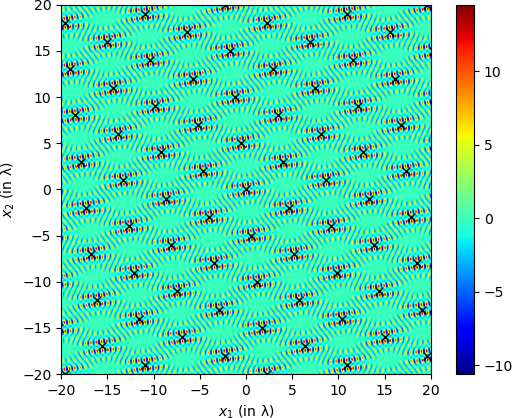} &
	\includegraphics[width=0.30\textwidth]{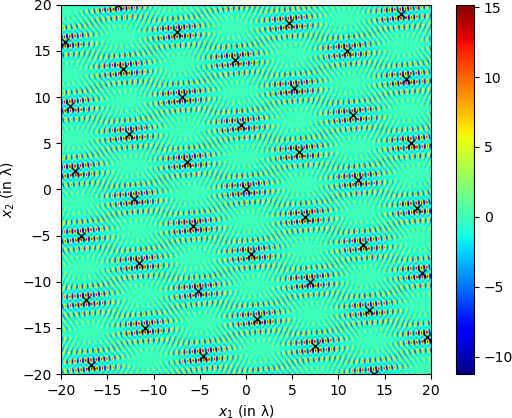}\\
	\includegraphics[width=0.30\textwidth]{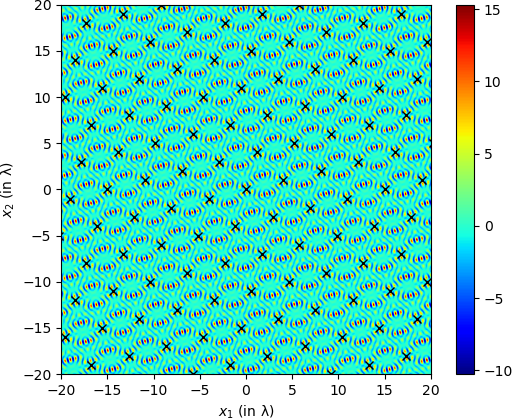} &
	\includegraphics[width=0.30\textwidth]{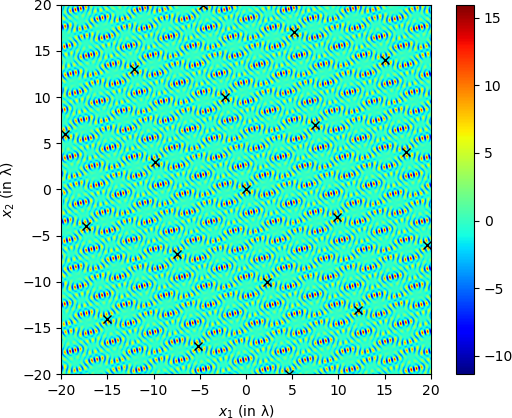} &
	\includegraphics[width=0.30\textwidth]{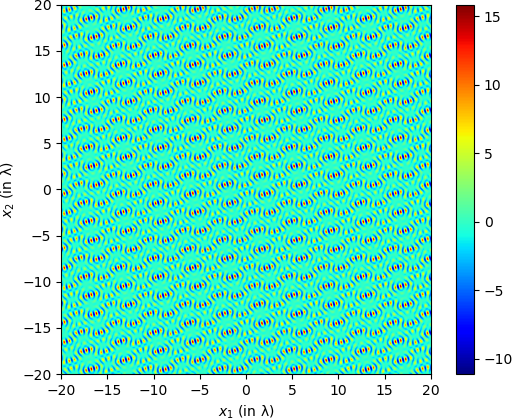}\\
	$m=2, r=3$ & $m=3, r=4$ & $\theta = \pi/6$
 \end{tabular}
 \caption{Examples of moir\'e patterns obtained as in \cref{ex:moire} by superposing two sets of plane waves with ARPs having hexagonal Bravais lattice rotated by angle $\theta(m,r)$. The superlattice points appear as $\times$. Here the transducer parameters are $\alpha_j=1, \beta_j = -1$, $j=1,\ldots, 4$. The bottom-right image corresponds to an angle $\theta=\pi/6$ which gives a quasiperiodic ARP.}
 \label{fig:moire}
\end{figure}

\end{example}

\subsection{Quasiperiodic tilings}
\label{sec:tilings}

An ARP for a quasiperiodic standing wave in dimension $d$ is associated with a quasiperiodic tiling in dimension $d$. Here, we show the tilings in $d=2$ for each of the ARPs we constructed in \cref{ex:arp}. To construct the tilings, we follow the projection method \cite{Austin:2005:PTT,Yamamoto:1996:CQC,deBruijn:1981:ATP}. The main idea is that the nodes of the tiling are orthogonal projections of the $N-$dimensional lattice points with associated Voronoi cells intersecting the plane. The edges in the tiling indicate whether the corresponding Voronoi cells are adjacent in higher dimensions.
An example showing this tiling method with $N=2$ and $d=1$ appears in \cref{fig:tiling-1D}. In 2D, the procedure to obtain tilings is as follows and can be easily adapted to e.g., 3D tilings. We emphasize that these quasiperiodic tilings are not the intersections of the Voronoi cells with the affine subspace we work on. In the example of \cref{fig:tiling-1D}, the intersections are line segments with variable length, whereas the tiling is composed of segments with two distinct lengths. Because the tilings are not intersections, it is not clear whether the values of the ARP are directly related to the tiling. In particular, it seems that the number of ARP extrema is variable within each tile.

\begin{algorithm}
\begin{enumerate}
\item[Step 1] Generate a list $L \subset \real^2$ of points in a given region of interest.
\item[Step 2] Generate a list $V$, with no repetitions, of the $N-$dimensional lattice points whose Voronoi cells contain the points in the list $L$. Since the Voronoi cells of a regular lattice in $N$ dimensions are $N-$dimensional cubes, this can be done by rounding.
\item[Step 3] Add to the list $V$ all the lattice points whose Voronoi cells are adjacent to the Voronoi cells corresponding to the points in $V$ and that intersect the affine subspace $\vy = \mK^T \vx + \vgamma$. This is done by solving \eqref{eq:linearprog}.
\item[Step 4] For each lattice point $\vv \in V$, find its orthogonal projection onto the affine subspace $\vy = \mK^T \vx + \vgamma$. This is done by solving a least squares problem with the normal equations: $\text{projection}(\vv)= (\mK\mK^T)^{-1}\mK(\vv - \vgamma)$.
\item[Step 5] Add an edge between two nodes if their corresponding Voronoi cells are adjacent.
\end{enumerate}
\caption{Computing the tiling of the plane corresponding to a quasiperiodic ARP.}
\label{algo:qptile}
\end{algorithm}

Finally, we explain how we determine when a Voronoi cell with center point $\vc = (c_1,\ldots,c_N) \in (2\pi)\znat^N$ intersects the affine subspace $\vy = \mK^T \vx + \vgamma$, which is used in \cref{algo:qptile}, Step 3. This is done by checking the feasibility of the linear programming problem
\begin{equation}
\min_{\vx \in \real^2}~C~\text{s.t.}~c_j - \pi \le \vk_j\cdot\vx + \gamma_j\le c_j + \pi,  j = 1,\ldots, N.
\label{eq:linearprog}
\end{equation}
Notice that the objective function is an arbitrary constant $C$, as we are attempting to solve a linear programming problem to check whether the feasible set determined by these linear constraints is non-empty. The Voronoi cell with center $\vc$ intersects the affine subspace $\vy= \mK^T \vx + \vgamma$ if and only if the feasible set in \eqref{eq:linearprog} is non-empty.

\begin{figure}
\begin{center}
	\includegraphics[width = 0.4\textwidth]{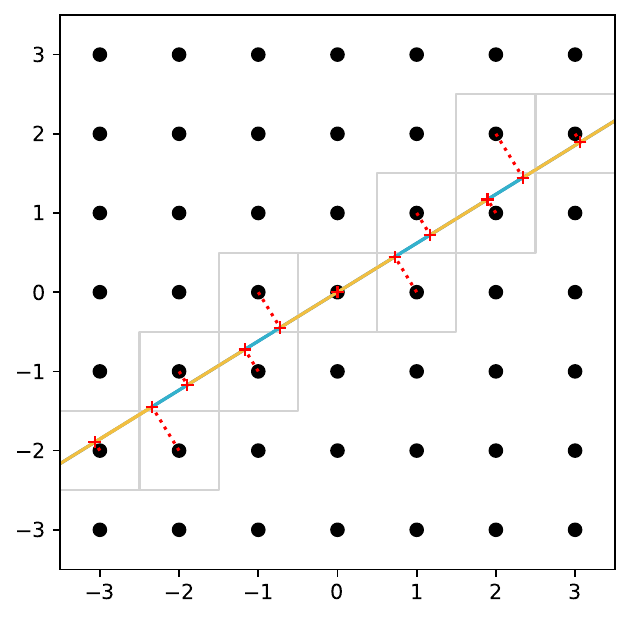}
\end{center}
	\caption{Example of tiling for $N=2$ and $d=1$. We start with the $\znat^2$ lattice (black dots) and the line $y = (1/\phi)x$, where $\phi = (1+\sqrt{5})/2$. The Voronoi cells of the lattice intersecting the line are highlighted in gray. The centers of these Voronoi cells are orthogonally projected onto the line (projections are shown in red). The projected points split the line into short (blue) and long (yellow) tiles. Because the line goes through the origin and has an irrational slope, it will never pass through another coordinate point; therefore, the pattern of short and long tiles is quasiperiodic -- if we continued the line tiling indefinitely, certain patterns would repeat (e.g., short-long-long), but the entire tiling would not have translational symmetry.}
	\label{fig:tiling-1D}
\end{figure}

\Cref{ex:tilings} shows eightfold, tenfold, and twelvefold tilings constructed using the method described in this section. To further illustrate the quasiperiodicity of the ARPs presented in \cref{ex:arp}, we show these tilings overlaid on top of the corresponding ARPs and notice that the tilings further emphasize the repeating, but non-periodic, structures present. 

\begin{example}
\label{ex:tilings}
\Cref{fig:tilings} shows the same ARPs shown in \cref{fig:arp}, but this time with their associated tilings overlaid. The tilings are obtained applying \cref{algo:qptile} with $N=4$, $5$ and $6$. The figures show that the tiles' values align with the ARP. 

The tiling rhombus angles can be determined by recognizing that the lines in an $N-$grid only intersect with angles that are multiples of $\pi/N$. This restriction means that the tiles for the eightfold symmetry case are of two kinds: squares and rhombuses with acute angle 45\textdegree. For the tenfold symmetry, the tiles are rhombuses with acute angles 36\textdegree~or 72\textdegree (these are known as Penrose tiles \cite{Penrose:1974:ROP,Penrose:1979:PCN,deBruijn:1981:ATP}). For the twelvefold symmetry case, the tiles are squares and two kinds of rhombuses with acute angles of 30\textdegree~and 60\textdegree.

\begin{figure}
\begin{center}
	\begin{tabular}{c@{}c@{}c}
	\includegraphics[width=0.3\textwidth]{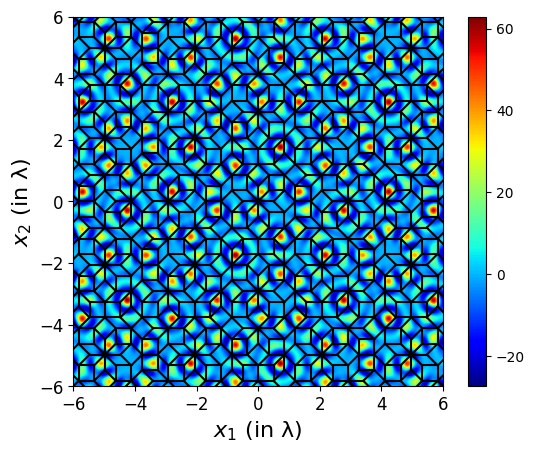} &
	\includegraphics[width=0.3\textwidth]{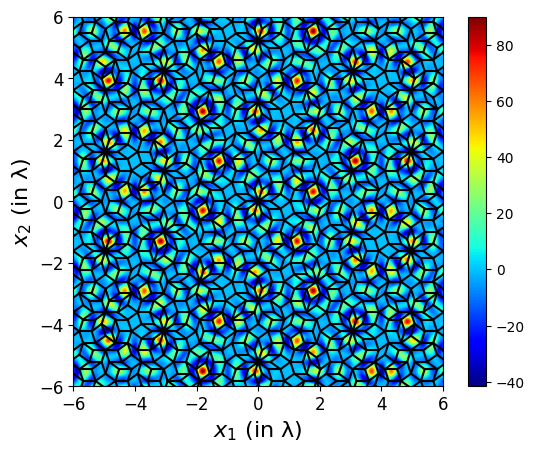} &
	\includegraphics[width=0.3\textwidth]{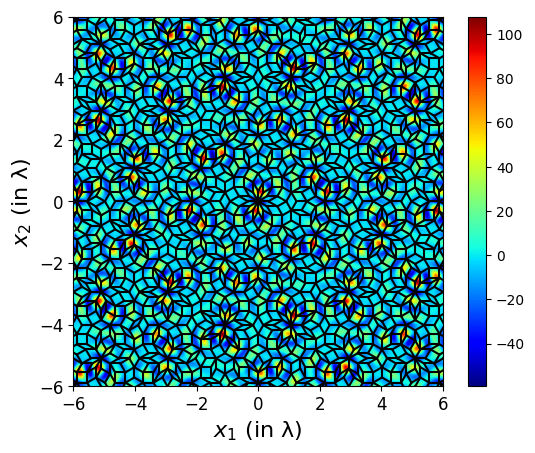} \\
	(a) & (b) & (c)
	\end{tabular}
\end{center}
\caption{Examples of quasiperiodic ARPs with associated quasiperiodic tilings superimposed. The ARPs are the same as in \cref{fig:arp}. (a) Eightfold tiling superimposed on ARP constructed with $N=4$ transducer directions. (b) Tenfold tiling superimposed on ARP constructed with $N=5$ transducer directions. (c) Twelvefold tiling superimposed on ARP constructed with $N=6$ transducer directions.}
\label{fig:tilings}
\end{figure}

\Cref{fig:tilings5} shows that changing the properties of the particles and fluids in the computation of the ARP does not change the underlying tiling. The ARP shown in \cref{fig:tilings5} (b) is the same ARP seen in \cref{fig:arp} (b) and \cref{fig:tilings} (b). Recall that this ARP was computed with parameters $\mathfrak{a}=\mathfrak{b}=1$. The ARPs in \cref{fig:tilings5} (a) and (c) were computed with parameters $\mathfrak{a}=1, \mathfrak{b}=0$ and $\mathfrak{a}=0, \mathfrak{b}=1$, respectively (without changing the wavevectors or transducer parameters). Although the ARPs look very different, the same tenfold tiling is visually correlated with each ARP.

\begin{figure}
\begin{center}
 \begin{tabular}{c@{}c@{}c}
  \includegraphics[width=0.295\textwidth]{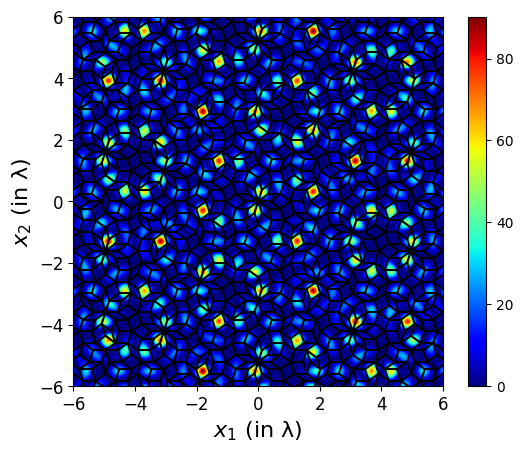} & 
  \includegraphics[width=0.3\textwidth]{tiling_10fold.png} &
  \includegraphics[width=0.3\textwidth]{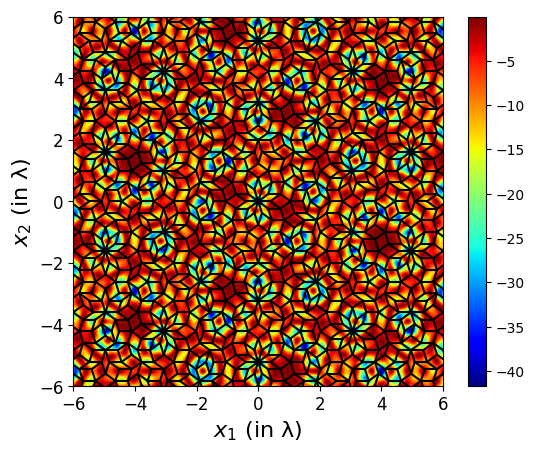}\\
  (a) & (b) & (c)
 \end{tabular}
\end{center}
 \caption{Examples of ARPs corresponding to different properties of the particles and the fluids that have exactly the same underlying tiling as in \cref{fig:tilings} (b). For all three ARPs, the wavevectors are $\vk_j = [\cos((j-1)\pi/N), \sin((j-1)\pi/N)]^T, j=1, \ldots, 5$, and the transducer parameters are $\alpha_j=1, \beta_j = -1$, $j=1,\ldots, 5$. The ARP parameters vary. (a) $\mathfrak{a}=1$ and $\mathfrak{b}=0$. (b) $\mathfrak{a}=\mathfrak{b}=1$. (c) $\mathfrak{a}=0$ and $\mathfrak{b}=1$.}
 \label{fig:tilings5}
\end{figure}

\end{example}

\section{Periodic acoustic radiation potentials}
\label{sec:derivatives}
Here we focus on periodic ARPs in dimension $N$, i.e., ARPs corresponding to
\eqref{eq:planewaves} where $\vk_j \in \real^N$ for $j=1,\ldots,N$. This is the
same setup that was considered in \cite{Guevara:2019:PPA}. Throughout this section, we omit the $d$ and $N$ subscripts on $p$ and $\psi$ since we are not considering any restrictions to lower-dimensional subspaces. Here, we provide explicit expressions for the spatial gradient and Hessian of the ARP (\cref{sec:gradhess}). This was not necessary in \cite{Guevara:2019:PPA} because we used a direct approach to characterize the minima corresponding to transducer parameters that are eigenvectors of the matrix $\mQ(\vzero)$ (see \eqref{eq:Qmat}). With the explicit expressions of the gradient and Hessian, we can verify that the spatial minima of the ARP in the particular case of \cite{Guevara:2019:PPA} satisfy the second-order necessary optimality conditions (\cref{sec:levelset}). We also identify in \cref{lem:indet} a periodic family of affine subspaces that must be contained in certain level-sets of the ARP, provided the transducer parameters are also eigenvectors of the matrix $\mQ(\vzero)$.

\subsection{Explicit formulas for spatial gradient and Hessian}
\label{sec:gradhess}

Here, we provide explicit expressions for the gradient and Hessian of an ARP of the form \eqref{eq:ARP2}.  The derivatives of ARPs restricted to affine subspaces are considered later in \cref{sec:in-subspace}.

\begin{theorem}
\label{thm:gradient}
The gradient and Hessian of $\psi(\vx_0;\vu)$ are given by
\begin{equation}
\nabla_{\vx}\psi(\vx_0;\vu) = 2[\mK,-\mK]\Im(\diag(\overline{\vu})\mQ(\vx_0)\vu)
\label{eq:grad}
\end{equation}
and
\begin{equation}
\nabla^2_{\vx} \psi(\vx_0;\vu) = 2[\mK,-\mK]\bigg[\diag(\overline{\vu})\mQ(\vx_0)\diag(\vu) 
-\Re[\diag(\overline{\vu})\diag(\mQ(\vx_0)\vu)]\bigg]
\begin{bmatrix}
\mK^T \\
-\mK^T
\label{eq:hess}
\end{bmatrix}.
\end{equation}
\end{theorem}

\begin{proof}
First, we note that the translation result from equation 2.5 in \cite{Guevara:2019:PPA} holds for our general definition of $\mQ(\vx)$ involving any Hermitian matrix $\mA$, i.e., $\mQ(\vx_0+\veps)$ and $\mQ(\vx_0)$ are related by 
\begin{equation}
\mQ(\vx_0 + \veps) = \exp[i\mD(\mK^T\veps)]^*\mQ(\vx_0)\exp[i\mD(\mK^T\veps)].
\end{equation}

Using the Taylor series, we can write as $\veps \to 0$ the following expansion for the diagonal matrix:
\begin{equation}
\exp[i\mD(\mK^T\veps)] = \mI + i\mD(\mK^T\veps) - \dfrac{1}{2}[\mD(\mK^T\veps)]^2 + o(\|\veps\|^2).
\end{equation}

Substituting this expansion into the expression for $\mQ(\vx_0+\veps)$, we can obtain a second-order expansion of $\mQ(\vx)$ around $\vx=\vx_0$. Then recall the expression for the ARP is
\begin{equation}
\psi(\vx_0+\veps;\vu) = \vu^*\mQ(\vx_0+\veps)\vu,
\end{equation}
so we can use the expansion for $\mQ(\vx)$ to obtain a second-order expansion for $\psi(\vx;\vu)$ around $\vx = \vx_0$.
We then find expressions for $\nabla_{\vx}\psi(\vx_0;\vu)$ and $\nabla^2_{\vx}\psi(\vx_0;\vu)$ by equating terms from this expansion with the terms in the Taylor expansion of $\psi(\vx;\vu)$ around $\vx_0$:
\begin{equation}
\psi(\vx_0+\veps;\vu) = \psi(\vx_0;\vu) + \veps^T\nabla_{\vx}\psi(\vx_0;\vu) + \frac{1}{2}\veps^T\nabla^2_{\vx}\psi(\vx_0;\vu)\veps + o(\|\veps\|^2).
\end{equation}
Detailed computations can be found in \cref{app:grad}.
\end{proof}

\subsection{Characterization of spatial minima and level-sets}
\label{sec:levelset}
We show that we can force the ARP to have a minimum at $\vx_0$ by choosing the vector of transducer parameters $\vu$ to be an eigenvector of $\mQ(\vx_0)$ associated with the minimum eigenvalue.  Then, in \cref{lem:indet}, we give a periodic family of affine subspaces that must be included in the level-sets of the ARP, provided that we use an eigenvector of $\mQ(\vx_0)$ as transducer parameters. We start with two lemmas showing, in particular, that choosing the transducer parameters $\vu$ to be an eigenvector of $\mQ(\vx_0)$ corresponding to its smallest eigenvalue guarantees that $\vx_0$ satisfies the second-order necessary optimality conditions for a minimizer.

\begin{lemma} If $\vu$ is an eigenvector of $\mQ(\vx_0)$ associated with any eigenvalue $\lambda$, then $\nabla_{\vx}\psi(\vx_0;\vu) = \vzero$.
\label{lem:zero-gradient}
\end{lemma}
\begin{proof}
Using the expression for $\nabla_{\vx}\psi(\vx_0;\vu)$ found in \cref{thm:gradient}, we have
\begin{align}
\nabla_{\vx} \psi(\vx_0;\vu) &= 2[\mK,-\mK]\Im(\diag(\overline{\vu}\lambda\vu) \\
&= 2[\mK,-\mK]\lambda\Im(\overline{\vu}\odot \vu) \\
&= \vzero.
\end{align}
\end{proof}

\begin{lemma}
If $\vu$ is an eigenvector of $\mQ(\vx_0)$ associated with the minimum eigenvalue $\lambda$, then $\nabla^2_{\vx}\psi(\vx_0;\vu)$ is positive semidefinite.
\label{lem:pos-semidef}
\end{lemma}

\begin{proof}
Using the expression for $\nabla^2_{\vx}\psi(\vx_0;\vu)$ found in \cref{thm:gradient} and the fact that $\mQ(\vx_0)\vu = \lambda \vu$, we have
\begin{align}
\nabla^2_{\vx}\psi(\vx_0;\vu) &= 2[\mK,-\mK]\bigg[\diag(\overline{\vu})\mQ(\vx_0)\diag(\vu) - \Re[\diag(\overline{\vu})\diag(\lambda \vu)]\bigg]
\begin{bmatrix}
\mK^T \\
-\mK^T
\end{bmatrix} \\
&= 2[\mK,-\mK]\bigg[\diag(\overline{\vu})\mQ(\vx_0)\diag(\vu) - \lambda \diag(\overline{\vu})\diag(\vu)\bigg]
\begin{bmatrix}
\mK^T \\
-\mK^T
\end{bmatrix} \\
&= 2[\mK,-\mK]\diag(\overline{\vu})\bigg[\mQ(\vx_0)-\lambda \mI\bigg]\diag(\vu)
\begin{bmatrix}
\mK^T \\
-\mK^T
\end{bmatrix}.
\end{align}
Since $\lambda$ is the smallest eigenvalue of $\mQ(\vx_0)$, the matrix $\mQ(\vx_0)-\lambda \mI$ is positive semidefinite. Thus for any $\vz\in \complex^N$,
\begin{align}
\vz^*[\nabla^2_{\vx}\psi(\vx_0;\vu)]\vz &= 
2\vz^*[\mK,-\mK]\diag(\overline{\vu})\bigg[\mQ(\vx_0)-\lambda \mI\bigg]\diag(\vu)
\begin{bmatrix}
\mK^T \\
-\mK^T
\end{bmatrix}\vz \\
&= 2\left(\diag(\vu)
\begin{bmatrix}
\mK^T \\
-\mK^T
\end{bmatrix}
\vz\right)^*
\bigg[\mQ(\vx_0)-\lambda \mI\bigg]\diag(\vu)
\begin{bmatrix}
\mK^T \\
-\mK^T
\end{bmatrix}\vz \\
&\ge 0.
\end{align}
\end{proof}

The next lemma allows us to systematically construct a periodic subset of the level-set $\{ \vx \in \real^N ~|~ \psi(\vx;\vu) = \lambda\}$ when ($\lambda,\vu$) is an eigenpair of $\mQ(\vzero)$. Because the field we are working with is periodic, we can expect the level sets to be periodic. The set of points we find depends on the invariance of the eigenspace of $\mQ(\vzero)$ with respect to sign changes in the entries of $\vu$. One application of this result is to predict locations in $\real^N$ that are global minima when we take $\lambda$ as the smallest eigenvalue of $\mQ(\vzero)$.
\begin{lemma}
 \label{lem:indet}
 Let $\vu = [\vv;\pm \vv]$ be a real unit norm eigenvector of $\mQ(\vzero)$ associated with eigenvalue
 $\lambda$. Let $Z = \{ j ~|~ v_j =0 \}$ and $E = \{ j ~|~ v_j \ne 0 \}$. Define the space $S_{\vzero} = \linspan\{\va_j ~|~ j\in Z \}$. Then we have
 \begin{equation}
 S_{\vn} = S_{\vzero} + \sum\limits_{i\in E} \frac{n_i\va_i}{2} \subset \{ \vx \in \real^N ~|~ \psi(\vx;\vu) = \lambda\},
 \end{equation}
for all $\vn \in \znat^{|E|}$ such that $[\vw; \pm \vw]$ is a $\lambda-$eigenvector of $\mQ(\vzero)$ satisfying $\vw_E = (-1)^{\vn} \vv_E$ and $\vw_Z = \vv_Z$. For us $\vv_E$ is the restriction of the vector $\vv$ to the set of indices $E$, i.e. $(\vv_E)_i = \vv_{E_i}$, where $i = 1,\ldots,|E|$ and $E_i$ is the $i-$th element of the ordered set $E$, and similarly for $\vw_Z$.
\end{lemma}

\begin{proof}
Let $\vz\in S_{\vn}$. Then $\vz = \vz_1 + \vz_2$, where 
\begin{equation}
\vz_1 \in \linspan\{\va_j ~|~ j\in Z\} = \{\vk_j ~|~ j \not\in Z\}^{\perp}
~\text{and}~
\vz_2 = \sum\limits_{i\in E} \frac{n_i\va_i}{2},
,\end{equation}
for some $\vn\in Z^{|E|}$ such that $[\vw; \pm \vw]$ is a $\lambda-$eigenvector of $\mQ(\vzero)$, where $\vw_E = (-1)^{\vn} \vv_E$ and $\vw_Z = \vv_Z$.

For $j\in Z$, we have by definition that $v_j=0$, so $(\exp[i\mK^T\vz])_j \vv_j=\vv_j=0$. For $j\not\in Z$, we have $\vk_j\cdot \vz_1 = 0$ and
\begin{equation}
\vk_j\cdot \vz_2 = \vk_j\cdot \left(\sum\limits_{i\in E} \frac{n_i\va_i}{2} \right) 
= \sum\limits_{i\in E} \frac{1}{2} n_i\vk_j\cdot \va_i \\
= \sum\limits_{i\in E} \frac{1}{2} n_i2\pi\delta_{ij} \\
= \pi n_j.
\end{equation}
Thus, for $j\not\in Z$ we have that
\begin{equation}
(\exp[i\mK^T\vz])_j\vv_j = (-1)^{n_j}\vv_j.
\end{equation}
Now we can see that
\begin{equation}
\exp[i[\mK,-\mK]^T\vz]\odot [\vv;\pm \vv] = [\vw; \pm \vw],
\end{equation}
where $\vw_E = (-1)^{\vn}\vv_E$ and $\vw_Z = \vv_Z$. Thus $\exp[i[\mK,-\mK]^T\vz]\odot [\vv;\pm \vv]$ is a $\lambda-$ eigenvector of $\mQ(\vzero)$, meaning that $\psi(\vz;\vu) = \psi(\vzero;\vu) = \lambda$.
\end{proof}

\begin{remark}
 The previous result is a generalization of Lemma 2.3 in
 \cite{Guevara:2019:PPA}.
\end{remark}

\section{Control of quasiperiodic ARPs}
\label{sec:in-subspace}
We now consider ARPs corresponding to a superposition of plane waves
\eqref{eq:planewaves} where the plane wave directions $\vk_j \in \real^d$,
$j=1,\ldots, N$, and $N > d$. We recall from \cref{sec:quasiperiodic} that for
certain $\mK$, we can guarantee such ARPs are quasiperiodic. We start in
\cref{sec:qp:min} by giving the second-order necessary conditions for the
minima. Then in \cref{sec:qp:ca}, we show that if we restrict ourselves to
transducer parameters with the same amplitude, we can translate the ARPs at will
and also understand the changes to the ARP to first order.

\subsection{Characterization of spatial minima}
\label{sec:qp:min}
First, recall that \cref{lem:zero-gradient} and \cref{lem:pos-semidef} give the gradient and Hessian for a periodic ARP. We can obtain the gradient and Hessian for the quasiperiodic ARP by restricting to the lower-dimensional affine subspace $\{\vy = \mK^T \vx  ~|~ \vx\in \real^d\}$ by applying the chain rule. Recalling that $\psi_d(\vx;\vu) = \psi_N(\mK^T\vx;\vu)$, the gradient and Hessian of $\psi_d(\vx;\vu)$ are 
\begin{equation}
\nabla_{\vx}\psi_d(\vx;\vu) = \mK\nabla_{\vy}\psi_N(\mK^T \vx;\vu)
\label{eq:grad-restricted}
\end{equation}
and
\begin{equation}
\nabla_{\vx}^2\psi_d(\vx;\vu) = \mK\nabla_{\vy}^2\psi_N(\mK^T \vx;\vu) \mK^T.
\label{eq:hess-restricted}
\end{equation}
Thus, if $\vx$ is a minimizer of $\psi_d(\vx;\vu)$, the second-order necessary optimality conditions imply that $\nabla_{\vx}\psi_d(\vx;\vu) = 0$ and $\nabla_{\vx}^2\psi_d(\vx;\vu)$ is positive semidefinite.

\subsection{Rotations and reflections}
\label{sec:rotations}
We study the rotations and reflections that are possible if the transducer directions
$\vk_1,\ldots,\vk_N$ are fixed ahead of time. 

To fix ideas, we begin first in 2D and assume the transducer directions are $\vk_j = [\cos((j-1)\pi/N), \sin((j-1)\pi/N)]$, $j=1,\ldots,N$. In this situation, the possible rotation angles are $\pm (j-1)\pi/N, j=1,\hdots,N$, where $N$ is the number of transducer directions. These are exactly the angles of rotational symmetry for regular $2N-$gons. These rotations can be obtained by replacing the original transducer parameters $\vu$ by $\mP\vu$, where $\mP$ is a permutation matrix that applies an appropriate cyclic permutation to the elements of $\vu$. For an explicit example of rotation, see \cref{ex:rot}.

Reflections are also possible along the lines of $x_2 = \tan[(j-1)\pi/(2N)], j=1,\hdots,N$. These are exactly the lines of reflective symmetry for regular $2N-$gons. These reflections can be obtained by replacing the original transducer parameters $\vu$ by $\mP\vu$, where $\mP$ is a permutation matrix that applies the appropriate reflection to the transducer parameters. For an explicit example of reflection, see \cref{ex:rot}

More generally, if $\mR$ is a $d \times d$ real unitary matrix, then we can see that
\begin{equation}
 p(\mR \vx;\vu) = \widetilde{p}(\vx; \vu),
\end{equation}
where $\widetilde{p}$ is obtained as $p$, but replacing $\mK$ with
$\widetilde{\mK} = \mR^T \mK$ (see \eqref{eq:planewaves}). If the unitary matrix
$\mR$ is such that $\widetilde{\mK} = \mK \mD \mP$, where $\mP$ is a
permutation and $\mD$ is a diagonal matrix with plus or minus ones, then we
should be able to transform the ARP by $\mR^T$ without changing the transducer
directions by simply permuting and adjusting the signs of the transducer
parameters. This could allow for certain rotations and reflections of 3D ARP
patterns (since we are not restricting ourselves to, e.g., cyclic permutations).

\begin{example}
\label{ex:rot}
Consider again an ARP constructed with $N=5$ transducer directions, $\vk_j = [\cos((j-1)\pi/N), \sin((j-1)\pi/N)]^T, j=1, \ldots, 5$. Here we take the ARP parameters $\mathfrak{a}=\mathfrak{b}=1$, and transducer parameters $\vu = [-1,-1,1,1,1,-1,-1,1,1,1]^T$. This ARP is shown in \cref{fig:rotation}(a).

With the tenfold symmetry in the transducer arrangement, there are ten possible angles of rotation (up to signs and modulo $2\pi$), all multiples of $\pi/5$. We will perform a rotation by $3\pi/5$. To do this, we need to replace $\vu$ by $\mP\vu$, where $\mP$ applies a cyclic permutation to $\vu$ such that the transducer parameter $\vu_1$, originally applied to the transducer with direction $\vk_1 = [1,0]$, is now applied to the transducer with direction $\vk_4 = [\cos(3\pi/5),\sin(3\pi/5)]$, and the other transducer parameters follow similarly (now being applied to the transducer that forms an angle of $3\pi/5$ with the original transducer). This permutation vector is $[4,5,6,7,8,9,10,1,2,3]$. The resulting rotated ARP is shown in \cref{fig:rotation}(b).

There are 20 possible lines of reflection for this ARP, all forming angles with the positive $x_1$ axis that are multiples of $\pi/2N$. We will perform a reflection across the line $x_2 = \tan(\pi/5)x_1$. To do this, we need to replace $\vu$ by $\mP\vu$, where $\mP$ applies a permutation to $\vu$ such that each transducer parameter $\vu_j$ is applied to the transducer that is the reflective image of the $j-$th transducer across the line $x_2 = \tan(\pi/5)x_1$. This permutation vector is $[3,2,1,10,9,8,7,6,5,4]$. The resulting reflected ARP is shown in \cref{fig:rotation}(c).

\begin{figure}
\begin{center}
\begin{tabular}{c@{}c@{}c}
	\includegraphics[width=0.3\textwidth]{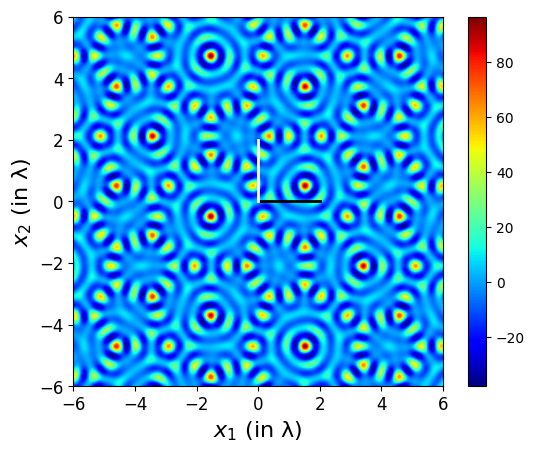} &
	\includegraphics[width=0.3\textwidth]{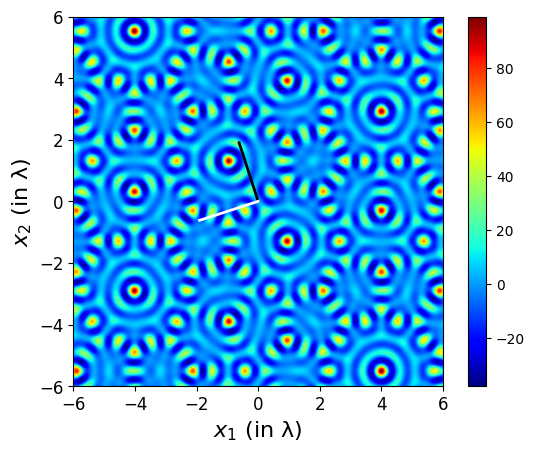} &
	\includegraphics[width=0.3\textwidth]{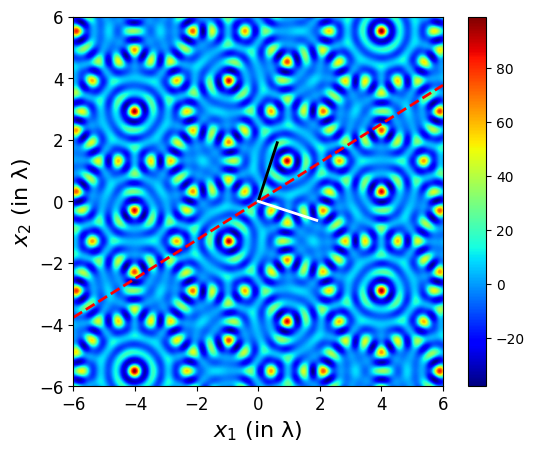} \\
	\includegraphics[width=0.1\textwidth]{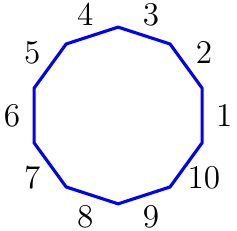} &
	\includegraphics[width=0.1\textwidth]{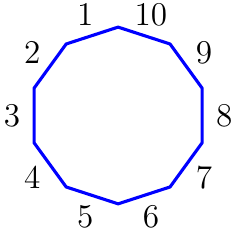} &
	\includegraphics[width=0.1\textwidth]{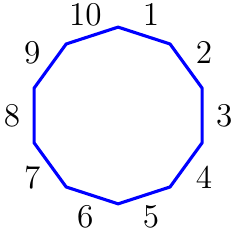} \\
	(a) & (b) & (c)
\end{tabular}
\end{center}
\caption{(a) Original ARP constructed with $N=5$ transducer directions. Here the ARP parameters are $\mathfrak{a}=\mathfrak{b}=1$ and the transducer parameters are $\vu = [-1,-1,1,1,1,-1,-1,1,1,1]^T$. The black and white line segments show the canonical basis vectors in $\real^2$ prior to any rotation or reflection. The decagon below has edges numbered according to the order of transducer parameters. (b) Rotation of (a) by $3\pi/5$. The basis vectors were also rotated. (c) Reflection of (a) across the line $x_2=\tan(\pi/5)x_1$, shown as a dashed red line. The basis vectors were also reflected. The decagons below show the necessary permutations of the transducer parameters.}
\label{fig:rotation}
\end{figure}

\end{example}
\subsection{Constant amplitude controls and spatial translations}
\label{sec:qp:ca}
At this point, we also recall from \cite{Guevara:2019:PPA} that a spatial translation of the ARP is equivalent to a constant amplitude change in the transducer parameters $\vu$, namely
\begin{equation}
\psi_N(\vy+\vh; \vu) = \psi_N(\vy; \exp[i\mD(\vh)]\vu),
\label{eq:transl}
\end{equation}

In the particular case where $\vh = \mK^T\veps$ for some $\veps \in \real^d$, we see that $\psi_d(\vx+\veps;\vu) = \psi_d(\vx;\exp[i\mD(\mK^T\veps)]\vu)$, by using \eqref{eq:transl} and $\psi_d(\vx;\vu) = \psi_N(\mK^T\vx;\vu)$. Thus, we can spatially translate a quasiperiodic ARP and its pattern of particles by simply adjusting the phases of the transducer parameters. In \cref{ex:translation}, we show one of the ARPs presented in \cref{fig:arp} being translated in this way.

\begin{example}
\label{ex:translation}

Consider the ARP from \cref{fig:arp}(b), where $N=5$ and $d=2$. We reproduce this ARP in \cref{fig:translation}(a). Next, we use \eqref{eq:transl} to translate this ARP in the plane. Letting $\vh = \mK^T\veps$ and $\veps = [\lambda,2\lambda]$ (where the columns of $\mK$ are the wavevectors from \cref{ex:arp}), we take the new transducer parameters to be $\exp[i\mD(\vh)]\vu$, and show the result in \cref{fig:translation}(b).

\begin{figure}
\begin{center}
\begin{tabular}{c@{}c}
\includegraphics[width = 0.45\textwidth]{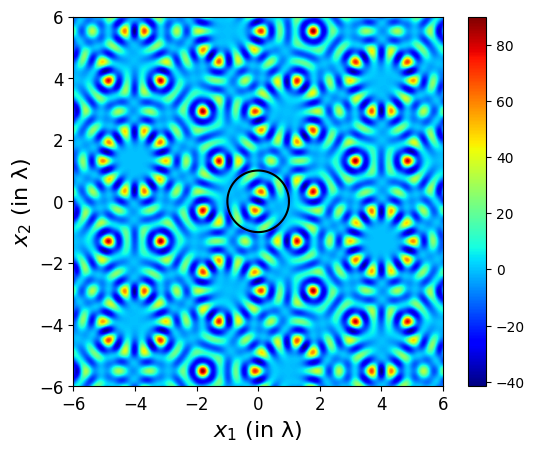} &
\includegraphics[width = 0.45\textwidth]{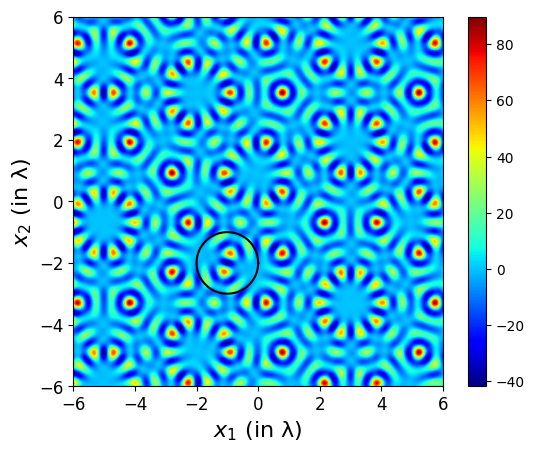} \\
(a) & (b)
\end{tabular}
\end{center}
\caption{The ARP shown in (a) is the same one shown in \cref{fig:arp} (b). The ARP in (b) was computed by changing the transducer parameters to $\exp[i\mD(\mK^T\veps)]\vu$, where $\vu$ are the transducer parameters from (a) and $\veps = [\lambda,2\lambda]$. The black circles show how the pattern was shifted. The pattern in the circle in (a) moves from the origin to $(-\lambda,-2\lambda)$ in (b), corresponding to evaluating the ARP function at $\vx + \veps$.}
\label{fig:translation}
\end{figure}

\end{example}

Naturally, we can also study what happens for other $\vh$ by linearization in $\real^N$. Indeed by using \eqref{eq:transl} and $\psi_d(\vx;\vu) = \psi_N(\mK^T\vx;\vu)$, we  can conclude that as $\epsilon \to 0$
\begin{equation}
\begin{aligned}
\psi_d(\vx;\exp[i\epsilon\mD(\vh)]\vu) = \psi_d(\vx;\vu) + \epsilon\nabla_{\vy}\psi_N(\mK^T \vx;\vu)^T \vh  + o(\epsilon).
\end{aligned}
\label{eq:phase:change}
\end{equation}
Thus the ARP with new transducer parameters $\exp[i\mD(\vh)]\vu$ (which amounts to only changing the phases), can be predicted to first order by calculating the gradient of the ARP $\psi_N$ using the explicit formula in \cref{sec:gradhess}.

In the next section, we show that if $\vh \in \range(\mK^T)^{\perp}$, the change in the ARP can be understood in the sense of changing the constraints in a constrained minimization formulation of the problem.

\section{Constrained optimization formulation}
\label{sec:constrained}
We would like to find local minima of the ARP only over a $2-$ or $3-$dimensional physical region. Written as a minimization problem, we need to solve
\begin{equation}
\min_{\vx\in \real^d} \psi_N(\mK^T\vx;\vu),
\end{equation}
where $d=2$ or 3. Notice that this could be reframed as the minimization of an $N-$dimensional field restricted to a $d-$dimensional subspace, where $d < N$:
\begin{equation}
\min_{\vy\in\real^N} \psi_N(\vy;\vu) \quad \text{s.t.} \quad \vy \in \{\mK^T\vx ~|~\vx \in \real^d\}.
\label{eq:constrained-minimization}
\end{equation}

Assuming $\mK$ has full rank (i.e., $\text{rank}(\mK) = d$), the constraint $\vy \in \{\mK^T\vx~|~\vx \in \real^d\}$ is equivalent to the constraint $\mZ^T\vy = 0$, where $\mZ$ is an $N\times (N-d)$ matrix satisfying $\range(\mZ) = \nullspace(\mK)$. This allows us to reformulate \eqref{eq:constrained-minimization} as the equality constrained optimization problem
\begin{equation}
\label{eq:constrained}
 \min_{\vy\in\real^N} \psi_N(\vy;\vu) \quad \text{s.t.} \quad \mZ^T\vy = 0.
\end{equation}

The next example illustrates the optimization problem \eqref{eq:constrained} for $N=2$, $d=1$.
\begin{example}
\label{ex:1D-constraint}
When $N=2$ and $d=1$, the constraint in \eqref{eq:constrained} is a line. In \cref{fig:1D}, we consider a transformation orthogonal to the constraint and show both the shifted constraint and the resulting transformed ARP function. 
\end{example}

\begin{figure}
\begin{center}
\begin{tabular}{c@{}c}
\includegraphics[width = 0.45\textwidth]{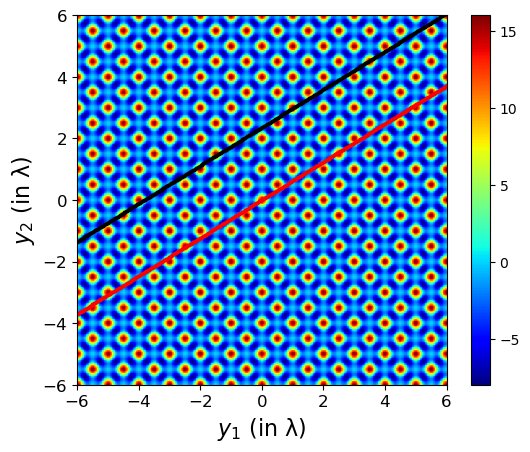} &
\includegraphics[width = 0.45\textwidth]{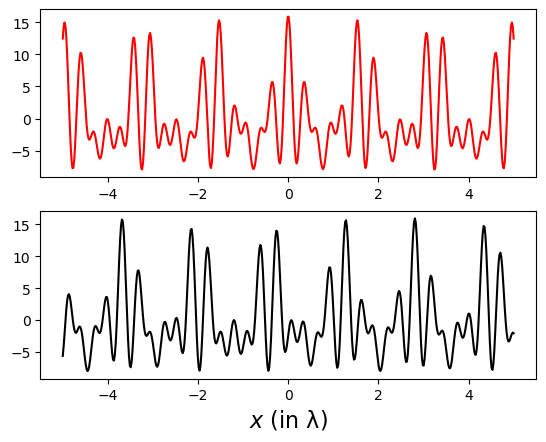} \\
(a) & (b)
\end{tabular}
\end{center}
\caption{(a) The function $g(\vy) = \psi(\vy;\vu)$, where $\vu = [1,1,1,1]^T$, the ARP parameters are $\mathfrak{a}=\mathfrak{b}=1$, and the transducer directions are $\ve_1,\ve_2$. (b) The function $f(x) = g(\mK^Tx+z\vh)$, where $\mK = [\phi,1]$, $\phi$ is the golden ratio, $\vh = [1,0]^T$, $z=0$ for the red plot, and $z=2\lambda$ for the black plot. The red and black plots correspond to evaluating the ARP shown in (a) along the red and black lines, respectively. See \texttt{1d\_fig\_anim.mp4} in the supplementary materials for an animation.}
\label{fig:1D}
\end{figure}

We start in \cref{sec:optimality} by applying the method of Lagrange multipliers to confirm the characterization of local minima given through other means in \cref{sec:qp:min}. The Lagrange multipliers are related to a projection of the gradient onto the subspace orthogonal to the constraints in \cref{sec:lagrange}. First order estimates of certain infinitesimally small phase changes of the transducer parameters are given in terms of the $N-$dimensional ARP (\cref{sec:dd}). These correspond to either translation of the ARP within the space of constraints or shifting the constraint subspace.

\subsection{Method of Lagrange multipliers}
\label{sec:optimality}
The Lagrange multiplier vector for the constrained optimization problem \eqref{eq:constrained} has length $N-d$ because there are exactly $N-d$ constraints. The associated Lagrangian function is
\begin{equation}
\mathcal{L}(\vy,\vlambda) = \psi_N(\vy;\vu) - \vlambda^T\mZ^T\vy.
\label{eq:lagrangian}
\end{equation}
We now show that the constrained minimization formulation of the problem corresponds to our first and second derivative results in \cref{lem:zero-gradient} and \cref{lem:pos-semidef} through the first- and second-order optimality conditions (also known as KKT conditions, see e.g. \cite{Nocedal:2006:NO}).

The \emph{first-order necessary condition} for optimality for \eqref{eq:constrained} is as follows. If  $\vy_0$ is a local minimum, then
\begin{equation}
\begin{aligned}
\nabla_{\vy}\psi_N(\vy_0;\vu) &= \mZ\vlambda,~\text{and}\\
\mZ^T\vy_0 &= 0.
\end{aligned}
\label{eq:kkt1}
\end{equation}
From \eqref{eq:kkt1}, we see that $\nabla_{\vy}\psi_N(\vy_0;\vu) \in \nullspace(\mK)$, so $\mK\nabla_{\vy}\psi_N(\vy_0;\vu) = 0$, which is exactly \eqref{eq:grad-restricted}.

The \emph{second-order necessary condition} for optimality for the minimization problem \eqref{eq:constrained} is as follows. If $\vy_0$ is a local minimum, then
\begin{equation}
\vs^T \nabla_{\vy}^2 \psi_N(\vy_0;\vu) \vs \ge 0 \qquad \text{for all~} \vs \in \nullspace(\mZ^T).
\label{eq:2nd-order-cond}
\end{equation}
Since $\nullspace(\mZ^T) = \range(\mK^T)$, we can rewrite \eqref{eq:2nd-order-cond} as
\begin{equation}
\vw^T \mK \nabla_{\vy}^2 \psi_N(\vy_0;\vu) \mK^T \vw \ge 0 \qquad \text{for all~} \vw \in \complex^d,
\end{equation}
which implies that $\mK \nabla_{\vy}^2 \psi_N(\vy_0;\vu) \mK^T$ is positive semidefinite, which is exactly \eqref{eq:hess-restricted}.

\subsection{Lagrange multipliers and the orthogonal space gradient}
\label{sec:lagrange}
The first-order necessary KKT condition \eqref{eq:kkt1} is satisfied at very particular points, including minimizers of \eqref{eq:constrained}. However, we can relax this condition to define a field $\vlambda(\vy)$ for all feasible $\vy = \mK^T \vx$, such that it agrees with the Lagrange multipliers at points that satisfy \eqref{eq:kkt1}. We can do this by expressing $\vlambda(\vy)$ as the solution to the linear least squares problem:
\begin{equation}
  \vlambda(\vy) = \argmin \| \nabla_{\vy} \psi_N(\vy;\vu) - \mZ \vlambda \|^2,
\end{equation}
for all $\vy = \mK^T \vx$ and fixed $\vu$. Under the assumptions on $\mZ$, the solution is given by
\begin{equation}
\vlambda(\vy) = (\mZ^T\mZ)^{-1}\mZ^T\nabla_{\vy} \psi_N(\vy;\vu).
\label{eq:lagrange-multipliers}                                     
\end{equation}  
We notice that $\mZ\vlambda(\vy)$ is the orthogonal projection of $\nabla_{\vy}
\psi_N(\vy;\vu)$ onto the subspace $\range(\mK^T)^{\perp} = \nullspace(\mK)= \range(\mZ)$. Furthermore, if we assume the columns of $\mZ$ are orthonormal i.e. $\mZ^T\mZ= \mI$, equation \eqref{eq:lagrange-multipliers} simplifies to
\begin{equation}
\vlambda(\vy) = \mZ^T\nabla_{\vy} \psi_N(\vy;\vu),
\label{eq:directional-derivative}
\end{equation}
showing that the Lagrange multipliers are components of the gradient of $\psi_N(\vy;\vu)$ in the basis formed by the columns of $\mZ$, which is a basis for the subspace orthogonal to the constraints.

\subsection{Infinitessimally small phase changes}
\label{sec:dd}
We now revisit the linearization \eqref{eq:phase:change} when the transducer parameter $\vu$ is changed (in phase) to $\exp[i\epsilon \mD(\vh)] \vu$ for a small $\epsilon$. There are two cases worthy of note here: $\vh \in \range(\mK^T) = \nullspace(\mZ^T)$ and $\vh \in \nullspace(\mK) = \range(\mZ)$.

If $\vh \in \range(\mK^T) = \nullspace(\mZ^T)$, we can find a $\vz$ such that $\vh = \mK^T \vz$. By the translation relation \eqref{eq:transl}, we have $\psi_d(\vx;\exp[i\epsilon \mD(\vh)] \vu) = \psi_N(\mK^T(\vx + \epsilon\vz) ; \vu) = \psi_d(\vx + \epsilon \vz;\vu)$. Thus a phase adjustment of the form $\exp[i\epsilon \mD(\mK^T \vz)]$ is equivalent in $\real^d$ to a translation by $\vz$. In $\real^N$, the phase adjustment amounts to a translation within $\range(\mK^T)$. By the linearization \eqref{eq:phase:change} we get as $\epsilon \to 0$ that
\begin{equation}
	\psi_d(\vx + \epsilon \vz;\vu) = \psi_d(\vx;\vu) + \epsilon \nabla_{\vy}\psi_N(\mK^T\vx;\vu)^T \mK^T \vz + o(\epsilon).
\end{equation}
Thus the sensitivity of the $d-$dimensional ARP to translation in the direction $\vz \in \real^d$ is given by $\mK \nabla_{\vy}\psi_N(\mK^T\vx;\vu) = \nabla \psi_d(\vx;\vu)$.

On the other hand, if $\vh \in \nullspace(\mK) = \range(\mZ)$ we can find a $\vq$ such that $\vh = \mZ \vq$ and the linearization \eqref{eq:phase:change} as $\epsilon \to 0$ gives:
\begin{equation}
\begin{aligned}
\psi_d( \vx ; \exp[i\epsilon\mD(\mZ \vq)]\vu) &= \psi_d( \vx; \vu) + \epsilon \nabla_{\vy} \psi_N(\mK^T \vx;\vu)^T \mZ\vq + o(\epsilon).
\end{aligned}
\end{equation}
Hence the sensitivity of the ARP to phase changes of the form $\vh = \mZ \vq$ (orthogonal to $\range(\mK^T)$) is given by $\mZ^T \nabla_{\vy} \psi_N(\mK^T \vx;\vu)$.
Intuitively, this corresponds to translating the constraint subspace $\range(\mK^T)$ in a normal direction in $N-$dimensional space. Moreover, the sensitivity agrees with the Lagrange multiplier \eqref{eq:directional-derivative} since $\vlambda(\mK^T \vx) = \mZ^T \nabla_{\vy} \psi_N(\mK^T \vx;\vu)$.

\begin{example}
\label{ex:dir-deriv}
Consider again the ARP from \cref{fig:arp}(b). The fields shown in (b) and (c) of \cref{fig:dir-deriv} show the sensitivity to phase changes for directions $\vh$ parallel to the plane (b) and orthogonal to the plane (c). For the parallel direction, we chose the first row of $\mK$. For the orthogonal direction, we found a vector in $\text{null}(\mK)$.
\end{example}

\begin{figure}
\begin{center}
\begin{tabular}{c@{}c@{}c}
\includegraphics[width = 0.3\textwidth]{arp_10fold.png} &
\includegraphics[width = 0.3\textwidth]{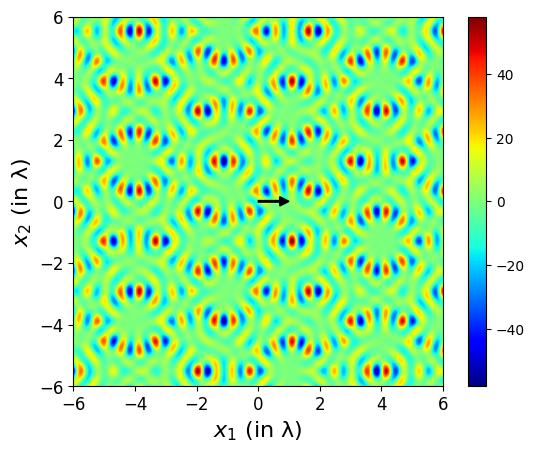} &
\includegraphics[width = 0.3\textwidth]{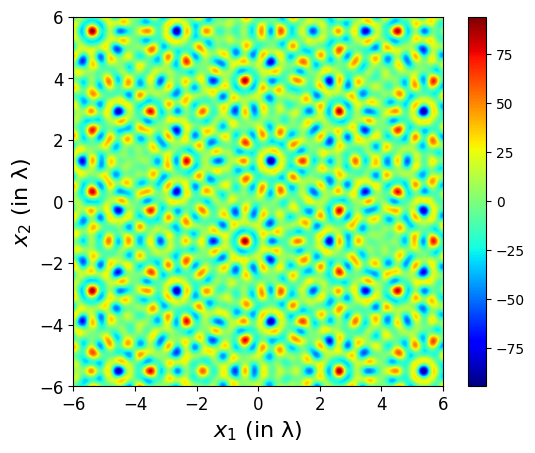} \\
(a) & (b) & (c)
\end{tabular}
\end{center}
\caption{(a) The ARP from \cref{fig:arp}(b). (b) 
The sensitivity of the ARP shown in the panel (a) (as calculated in \cref{sec:dd})
to multiplication by $\exp[i\mD(\vh)]$, with $\vh = [\cos(\pi/5),\cos(2\pi/5)$, $\cos(3\pi/5),\cos(4\pi/5),\cos(5\pi/5)]^T$ (the first row of $K$). Multiplication by $\exp[i\mD(\vh)]$ results in a translation 
indicated with a black arrow. (c) The sensitivity of (a) with $\vh = [1,-1,1,-1,1]^T \in \nullspace(\mK)$.}
\label{fig:dir-deriv}
\end{figure}

\section{Constant power changes}
\label{sec:power}
We would like to design a continuous transition of the particle structure 
between the patterns resulting from the ARPs corresponding to $\psi_d(\vx;\vu_0)$ and $\psi_d(\vx;\vu_1)$. We recall that $\psi_d(\vx; \vu) = \psi_N (\mK^T \vx; \vu)$, where $\psi_N$ is given as in \eqref{eq:ARP-N} and $\vu \in \complex^{2N}$ are the ``transducer parameters'' or coefficients used as weights in the linear combination of plane waves \eqref{eq:planewaves}. We assume that both transducer parameters correspond to the same power, i.e., without loss of generality $\|\vu_0\| = \|\vu_1\|=1$. It is reasonable to impose that any transducer parameters in the transition satisfy the same constant power constraint. We model this transition by a smooth open curve $\Gamma$ contained in the $\complex^{2N}$ unit sphere, with endpoints $\vu_0$ and $\vu_1$. 

To give a strategy for choosing such transition curves, we define the \emph{total ARP} associated with a region $\cA \subset \real^d$ with positive Lebesgue measure\footnote{The assumption of positive Lebesgue measure is unnecessary. A similar averaged ARP can be defined on zero measure lower dimensional subsets of $\real^d$. For example, in $d=2$, a similar cost functional has been used for e.g., finite point collections or curves \cite{Greenhall:2015:UDS}.} by the average of the ARP in the region $\cA$:
\begin{equation}
	\psi_{\cA}(\vu) = \frac{1}{|\cA|}\int_{\cA} d\vx\, \psi_d(\vx;\vu).
	\label{eq:total:arp}
\end{equation}
Since $\psi_d(\vx;\vu)$ is a quadratic function of $\vu$
for a Hermitian matrix depending on $\vx$, the same applies by linearity to $\psi_{\cA}(\vu)$:
\begin{equation}
	\psi_{\cA}(\vu) = \vu^* \left( \frac{1}{|\cA|} \int_{\cA} d\vx\; \mQ(\mK^T \vx)\right) \vu,
\end{equation}
where $\mQ(\vx)$ is defined as in \eqref{eq:Qmat}. Using the bound (see \cite{Guevara:2019:PPA})
\begin{equation}
	\lambda_{\min}(\mQ(\vzero))\leq \psi_d(\vx;\vu) \leq \lambda_{\max}(\mQ(\vzero)),
\end{equation}
we deduce a similar bound for the total ARP
\begin{equation}
	\lambda_{\min}(\mQ(\vzero)) \leq \psi_{\cA}(\vu) \leq \lambda_{\max}(\mQ(\vzero)).
	\label{eq:psibound}
\end{equation}
We think of $\psi_{\cA}(\vu)$ as an energy even if it may be negative since $\mQ(\vzero)$ may be indefinite. 
We define the ``cost'' $C(\Gamma,\cA)$ of the transition along the curve $\Gamma$ as
\begin{equation}
	C(\Gamma,\cA) = \int_{\Gamma} dl(\vu) \psi_{\cA}(\vx;\vu).
	\label{eq:cost}
\end{equation}
One strategy to choose $\Gamma$ is to find a curve $\Gamma$ that minimizes the cost $C(\Gamma,\cA)$ under the constraints that $\Gamma$ lies on the $\complex^{2N}$ unit sphere and the endpoints of $\Gamma$ are $\vu_0$ and $\vu_1$. A simpler (but not equivalent) strategy is to notice that the bound \eqref{eq:psibound} implies 
\begin{equation} 
	C(\Gamma,\cA) \leq |\Gamma| \lambda_{\max}(\mQ(\vzero)),
	\label{eq:costbound}
\end{equation} 
where $|\Gamma|$ is the length of the curve $\Gamma$ under the standard $\complex^{2N}$ metric. Therefore, we may also take $\Gamma$ to be a geodesic on the sphere $\complex^{2N}$ joining $\vu_0$ and $\vu_1$, or in other words, a curve $\Gamma$ with endpoints $\vu_0$ and $\vu_1$ with minimal arc length. Two advantages of this strategy are that the geodesics on a sphere have explicit expressions, and the result is independent of the region of interest $\cA$. To find the geodesics, it is convenient to use the standard isomorphism between $\complex^{2N}$ and $\real^{4N}$, i.e., we identify $\vu \in \complex^{2N}$ with $\widetilde{\vu} \in \real^{4N}$ defined by $\widetilde{\vu} = [ \vu'; \vu'']$, where $\vu = \vu' + i \vu''$ and $\vu', \vu'' \in \real^{2N}$. Notice that this isomorphism preserves lengths since $\|\vu\|^2 = \|\vu'\|^2 + \|\vu''\|^2 = \|\widetilde{\vu}\|^2$. In this section, we use tilde to denote the purely real quantity associated with a complex quantity via this isomorphism.

The set of all transducer parameters $\vu$ with power equal to one is the unit sphere in $\complex^{2N}$, i.e. 
\begin{equation}
	S_{2N} = \{ \vu \in \complex^{2N} ~|~ \|\vu\| = 1\}.
\end{equation}  
The unit sphere $S_{2N}$ in $\complex^{2N}$ can be identified to the unit sphere in $\real^{4N}$, i.e.
\begin{equation} 
	\widetilde{S}_{4N} = \{ \widetilde{\vu} \in \real^{4N} ~|~ \|\widetilde{\vu}\| = 1\}.
\end{equation} 
In particular, a curve $\Gamma \subset S_{2N}$ joining $\vu_0$ and $\vu_1$ can be identified with a curve $\widetilde{\Gamma} \subset \widetilde{S}_{4N}$ joining $\widetilde{\vu}_0$ and $\widetilde{\vu}_1$. Moreover, the two curves have the same length, i.e., $|\Gamma| = |\widetilde{\Gamma}|$. Hence a geodesic $\Gamma_{\text{geodesic}}$ in $S_{2N}$ joining $\vu_0$ and $\vu_1$ can be identified to a geodesic $\widetilde{\Gamma}_{\text{geodesic}}$ in $\widetilde{S}_{4N}$ joining $\widetilde{\vu}_0$ and $\widetilde{\vu}_1$. As can be found in e.g. \cite{doCarmo:1992:RG}, a parametrization of $\widetilde{\Gamma}_{\text{geodesic}}$ by its arc length is 
\begin{equation}
	\widetilde{\Gamma}_{\text{geodesic}} = \left\{ \widetilde{\vu}_0 \cos(t) + \widetilde{\vv} \sin(t) ~\Big|~ t \in \left[0, |\widetilde{\Gamma}_{\text{geodesic}}|\right]\right\},
	\label{eq:geodesic}
\end{equation}
where $\widetilde{\vv} = (\widetilde{\vu}_1-\widetilde{\vu}_0) / \|\widetilde{\vu}_1-\widetilde{\vu}_0\|$ and $|\widetilde{\Gamma}_{\text{geodesic}}| = \arccos(\widetilde{\vu}_0 \cdot \widetilde{\vu}_1)$. An example of a geodesic joining two points is illustrated in \cref{fig:geodesic} for the unit sphere in $\real^3$.

\begin{figure}
	\centering
	\includegraphics[width=0.3\textwidth]{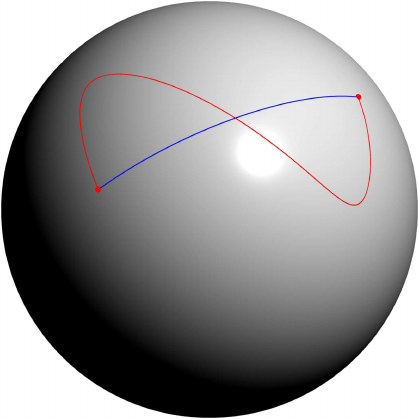}
	\caption{Both blue and red paths are smooth paths joining the same points on the unit sphere in $\real^3$. The blue path is a geodesic minimizing the arc length.}
	\label{fig:geodesic}
\end{figure}

\begin{example}\label{ex:geodesic}
We consider two ways of smoothly transitioning from the ARP in \cref{fig:translation}(a) where $N = 5$, $\lambda = 2\pi$, $\mathfrak{a}=\mathfrak{b}=1$ and $$\vu_0 = [1,1,1,1,1,-1,-1,-1,-1,-1],$$ to the ARP in \cref{fig:translation}(b) where $\vu_1 = \exp[i\mD(\mK^T\veps)] \vu_0$, with $\veps = [\lambda,2\lambda]$. We recall this transformation amounts to a translation of the ARP pattern, where the new origin is at $-\veps$. We consider two constant power paths joining $\vu_0$ and $\vu_1$: a ``direct path'' which can be interpreted as a translation, and a ``geodesic path'' which is the shortest path on the $\complex^{10}$ unit sphere.

\textbf{Direct path:} We define the curve $\Gamma_{\text{direct}} = \{ \vgamma_{\text{direct}}(t) ~|~ t \in [0,1] \}$ where $\vgamma_{\text{direct}}(t) = \exp[it\mD(\mK^T\veps)]\vu_0$. By construction, the endpoints of $\Gamma_{\text{direct}}$ are $\vu_0$ and $\vu_1$ and $\|\vgamma(t)\|=1$, so this curve lies on the $\complex^{10}$ unit sphere. Moreover its derivative with respect to $t$ is constant: $\dot{\vgamma}(t) = i \mD(\mK^T\veps) \vu_0$. Therefore, the length of the curve is given by $|\Gamma_{\text{direct}}| = \|\dot{\vgamma}(1)\| = \|\mD(\mK^T\veps)\vu_0\| =  \pi\sqrt{10} \approx 9.9$. ARPs evaluated on $[-6\lambda,6\lambda]^2$ for different $t \in [0,1]$ are given in \cref{fig:geochange} top row, which shows a constant velocity translation of the pattern. For an animation, see \texttt{anim\_translation.mp4} in the supplementary materials).

\textbf{Geodesic path:} We define the geodesic curve $\tilde{\Gamma}_{\text{geodesic}}$ joining $\widetilde{\vu}_0$ and $\widetilde{\vu}_1$ by using \eqref{eq:geodesic} on the $\real^{20}$ unit sphere. By the standard isomorphism, this defines a geodesic $\Gamma_{\text{geodesic}}$ joining $\vu_0$ and $\vu_1$ on the $\complex^{10}$ unit sphere. The length of these curves is $|\tilde{\Gamma}_{\text{geodesic}}| = |\Gamma_{\text{geodesic}}| = \arccos(\vw_0 \cdot \vw_1) \approx 1.4$, 
which is a significant decrease in length compared with the direct path in the previous case. 
The ARPs evaluated on $[-6\lambda,6\lambda]^2$ corresponding to different $t \in [0,|\Gamma_{\text{geodesic}}|]$ are given in \cref{fig:geochange} bottom row. For an animation see \texttt{anim\_translation\_geodesic.mp4} in the supplementary materials.

As expected, the geodesic path is shorter and by \eqref{eq:costbound} we expect to also reduce the total ARP in a region $\cA$.
\end{example}

\begin{figure}
	\centering
	\begin{tabular}{c@{\hspace{0.1em}}c@{\hspace{0.3em}}c@{\hspace{0.3em}}c@{\hspace{0.3em}}c@{\hspace{0.3em}}c@{\hspace{0.3em}}c}
		& $\vu_0$ & & & & & $\vu_1$ \\
		\raisebox{1em}{\rotatebox{90}{direct}} &
		\includegraphics[width=0.15\textwidth]{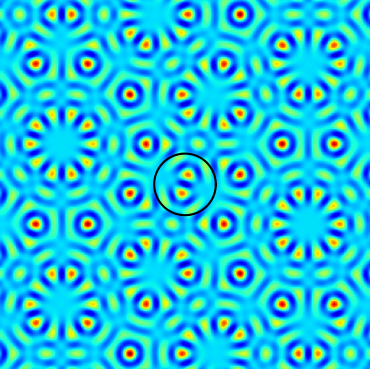} &
		\includegraphics[width=0.15\textwidth]{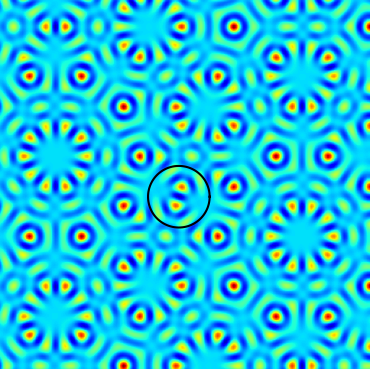} &
		\includegraphics[width=0.15\textwidth]{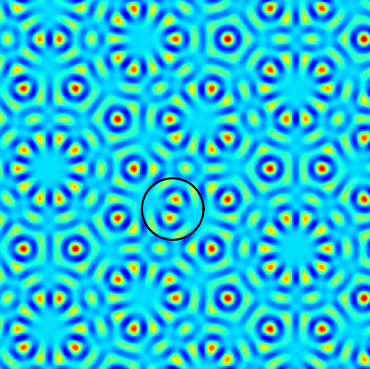} &
		\includegraphics[width=0.15\textwidth]{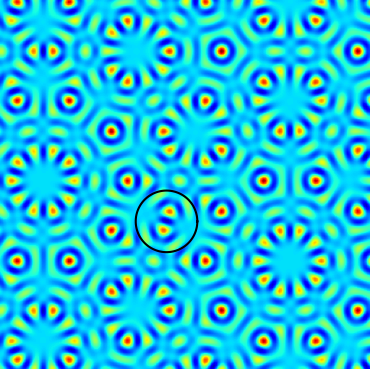} &
		\includegraphics[width=0.15\textwidth]{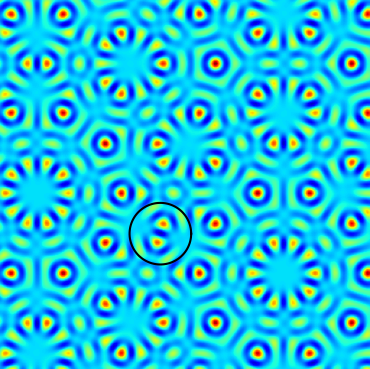} &
		\includegraphics[width=0.15\textwidth]{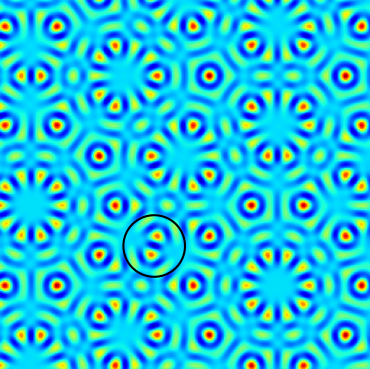} \\
		\raisebox{1em}{\rotatebox{90}{geodesic}} &
		\includegraphics[width=0.15\textwidth]{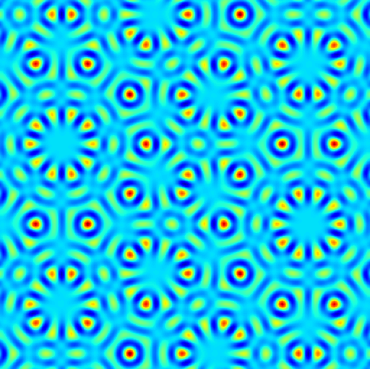} &
		\includegraphics[width=0.15\textwidth]{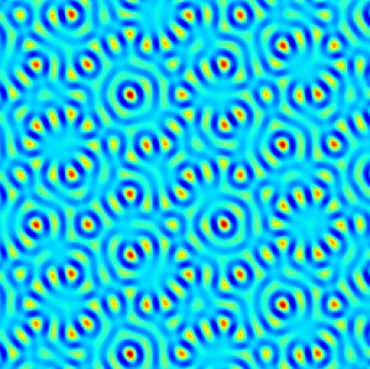} &
		\includegraphics[width=0.15\textwidth]{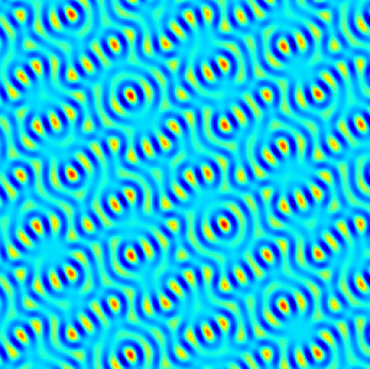} &
		\includegraphics[width=0.15\textwidth]{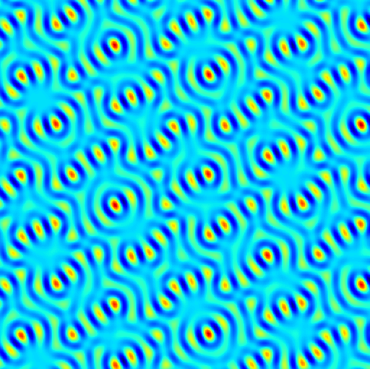} &
		\includegraphics[width=0.15\textwidth]{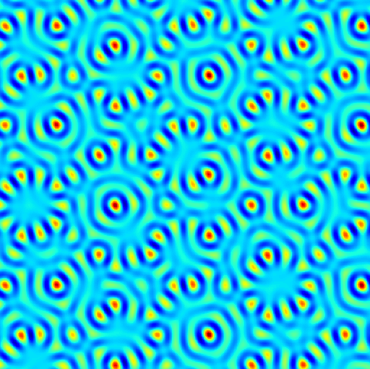} &
		\includegraphics[width=0.15\textwidth]{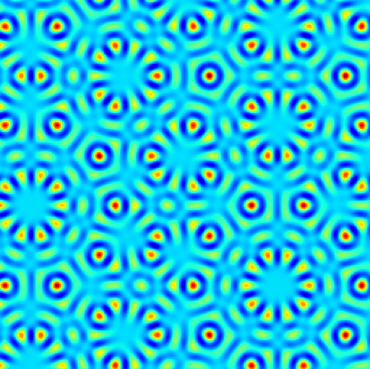} \\
	\end{tabular}
	\caption{Two ways of achieving a translation of the ARP configuration in \cref{fig:translation}(a) to that in \cref{fig:translation}(b). In the top row, each ARP configuration is translated by a smaller amount, as can be seen with the black circle that is included for reference. In the bottom row, we use a geodesic connecting the transducer parameters corresponding to the initial and final ARPs. The arc length between two consecutive frames on the top (resp. bottom) row is about $2.0$ (resp. $0.3$). See \cref{ex:geodesic} for the precise simulation parameters. For animations of these transitions see (a) \texttt{anim\_translation.mp4} and (b) \texttt{anim\_translation\_geodesic.mp4} in the supplementary materials.}
	\label{fig:geochange}
\end{figure}

\begin{example}
	We use the geodesic method to find a smooth transition between the ARPs in \cref{fig:rotation}(a) and its rotation by $3\pi/5$ about the origin, as in \cref{fig:rotation}(b). The transition is illustrated in \cref{fig:rotation:transition} and animated in \texttt{anim\_rotation.mp4} in the supplementary materials. In this particular case, there is no transition between these two ARPs consisting of only rotations since the rotation method in \cref{sec:rotations} is limited to a discrete set of angles.
\end{example}

\begin{figure}
	\centering
	\begin{tabular}{c@{\hspace{0.3em}}c@{\hspace{0.3em}}c@{\hspace{0.3em}}c@{\hspace{0.3em}}c@{\hspace{0.3em}}c}
		 $\vu_0$ & & & & & $\vu_1$ \\
		\includegraphics[width=0.15\textwidth]{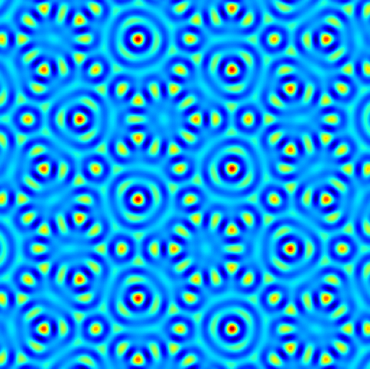} &
		\includegraphics[width=0.15\textwidth]{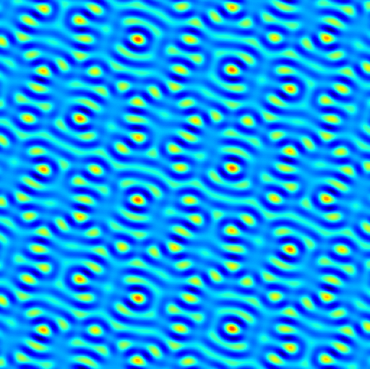} &
		\includegraphics[width=0.15\textwidth]{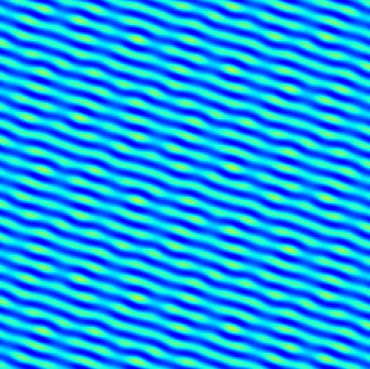} &
		\includegraphics[width=0.15\textwidth]{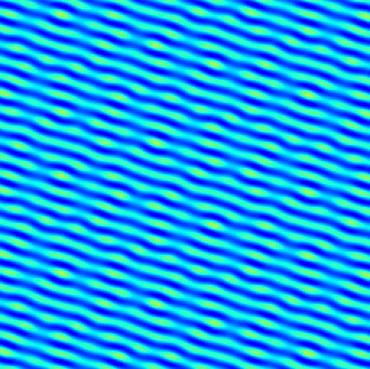} &
		\includegraphics[width=0.15\textwidth]{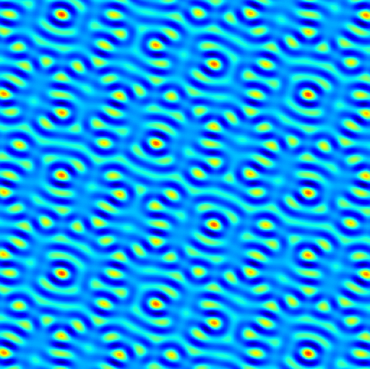} &
		\includegraphics[width=0.15\textwidth]{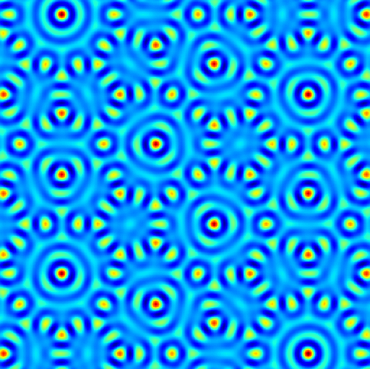} \\
	\end{tabular}
	\caption{Smooth transition between the ARP in \cref{fig:rotation}(a) and its rotated version \cref{fig:rotation}(b), using the geodesic method. See \cref{ex:rot} for the simulation parameters. For an animation see \texttt{anim\_rotation.mp4} in the supplementary materials.}
	\label{fig:rotation:transition}
\end{figure}

\section{3D examples}
\label{sec:num}
We now show examples of quasiperiodic ARPs in 3D, where the ultrasound transducers no longer need to be coplanar, and the wavevectors are in $\real^3$. The same condition on the matrix of wavevectors $\mK$ given in \cref{lem:condition} guarantees quasiperiodicity in 3D. But with this extra degree of freedom, we can construct ARPs that are periodic along only a particular direction and yet are not periodic in 3D (\cref{ex:3D_ex1}). We also show a quasiperiodic ARP in \cref{ex:3D_ex2} with icosahedral symmetry. 

\begin{example}[3D ARP that is periodic along only one direction]
\label{ex:3D_ex1}
We construct a 3D ARP that is $2\pi-$periodic along the $\ve_3$ direction and whose values at the planes $\vx_3 = 2n \pi$, $n\in \znat$, are identical to the 2D ARP in \cref{fig:arp}(b) plus a constant. Since this ARP is periodic in only one direction, we have that $\dim(\range(\mK^T) \cap (2\pi) \znat^n) = 1$. 
To construct this ARP, we take 
$\vk_j = [\tilde{\vk_j},0], j=1,\hdots,5$, where $\tilde{\vk_j}$ are the 2D wavevectors from \cref{ex:arp} and $\vk_6 = [0,0,1]^T$.

With these wavevectors, we construct the 3D ARP using \eqref{eq:ARP1} and \eqref{eq:planewaves}. For transducer parameters we use $\alpha_j=-\beta_j=1$, $j=1\ldots,6$ and for ARP parameters $\mathfrak{a}=\mathfrak{b}=1$. We can show that this 3D ARP restricted to a slice is equivalent to a 2D ARP plus a field depending on only the imaginary part of the 2D pressure and the $x_3$ value for the slice. Here we use the notation $\tilde{\vx} = [x_1,x_2]$ and $\tilde{\vu}$ is the vector of parameters $\alpha_j = -\beta_j = 1$, $j=1,\ldots, 5$. First, we write the 3D pressure $p_3(\vx;\vu)$ in terms of the 2D pressure $p_2(\tilde{\vx};\tilde{\vu})$. A simple calculation using \eqref{eq:planewaves} shows
\begin{equation}
p_3(\vx;\vu) = p_2(\tilde{\vx};\tilde{\vu}) + 2i\sin(x_3).
\end{equation}

The gradient can be split up similarly:
\begin{equation}
\nabla p_3(\vx;\vu) =
\begin{bmatrix}
\nabla p_2(\tilde{\vx};\tilde{\vu}) \\ 
2i\cos(x_3)
\end{bmatrix}.
\end{equation}

Then using \eqref{eq:ARP1}, the 3D ARP $\psi_3$ can be written in terms of the 2D ARP $\psi_2$:
\begin{equation}
\psi_3(\vx;\vu) = \psi_2(\tilde{\vx};\tilde{\vu}) + 4a\sin^2(x_3) - 4b\cos^2(x_3) + 4a\sin(x_3)\Im p_2(\tilde{\vx};\tilde{\vu}).
\label{eq:2.5d}
\end{equation}

We visualize different ``slices'' of the resulting 3D ARP and see that within any slice that is parallel to the original $x_3 = 0$ plane, the pattern is quasiperiodic, but the pattern has a period of $\lambda = 2\pi$ in the $x_3$ direction, as predicted by \eqref{eq:2.5d}. Several slices of the ARP are shown in \cref{fig:3D_ex1}. For an animation see \texttt{anim\_2\_5d.mp4} in the supplementary materials.

\begin{figure}
\begin{tabular}{cc}
\includegraphics[width=0.45\textwidth]{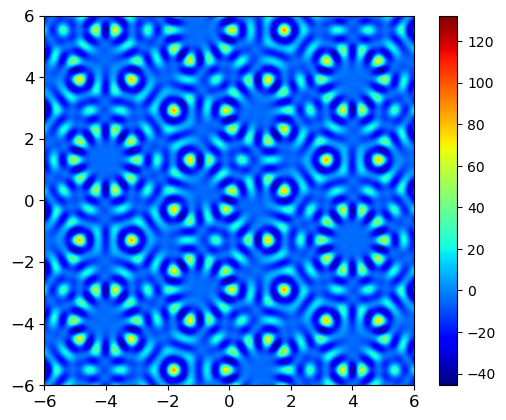} 
& \includegraphics[width=0.45\textwidth]{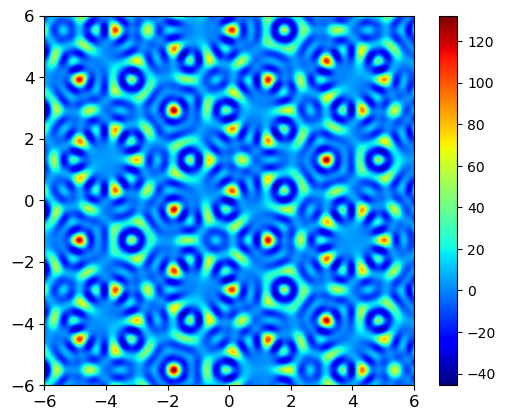} \\
$x_3 = 0$ & $x_3 = \lambda/4$ \\
\includegraphics[width=0.45\textwidth]{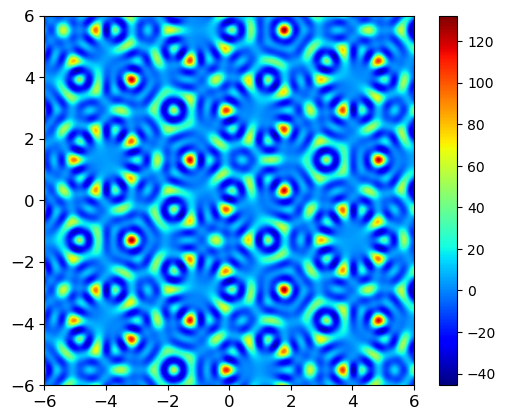}
& \includegraphics[width=0.45\textwidth]{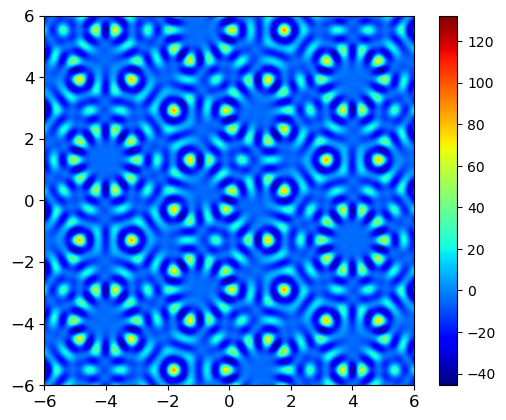} \\
$x_3 = 3\lambda/4$ & $x_3 = \lambda$
\end{tabular}
\caption{2D slices of a 3D ARP constructed to be periodic along only the $x_3$ direction. Within each planar slice, the ARP is quasiperiodic. As we move in the $x_3$ direction, the ARP repeats itself with period $\lambda = 2\pi$ (see \eqref{eq:2.5d}). Hence, (a) and (d) are identical. The wavevectors are $\vk_j = [\cos((j-1)\pi/N), \sin((j-1)\pi/N)]^T$, $j=1,\ldots,5$, and $\vk_6 = [0,0,1]^T$. The transducer parameters are $\alpha_j=-\beta_j=1$, $j=1,\ldots,6$. The ARP parameters are $\mathfrak{a}=\mathfrak{b}=1$. For an animation see \texttt{anim\_2\_5d.mp4} in the supplementary materials.}
\label{fig:3D_ex1}
\end{figure}
\end{example}

\begin{example}[3D quasiperiodic ARP with icosahedral symmetry]
\label{ex:3D_ex2}
To obtain a 3D quasiperiodic ARP with icosahedral symmetry, we take $N=6$ and $d=3$. We take the wavevectors to be the six directions that define an icosahedron:
\begin{equation}
\mK =\frac{1}{\sqrt{1+\phi^2}}\begin{bmatrix} 
	 0    & 1    & \phi & 0     & -1   & \phi \\
     1    & \phi & 0    & 1     & \phi & 0    \\
     \phi & 0    & 1    & -\phi & 0    & -1
\end{bmatrix},
\label{eq:icosa}
\end{equation}
where $\phi = (1+\sqrt{5})/2$ is the golden ratio.

\Cref{fig:icosahedron} shows three views of global minima of a 3D ARP constructed to have icosahedral symmetry. The second and third views correspond to projections that have fourfold (\cref{fig:icosahedron}(b)) and tenfold symmetry (\cref{fig:icosahedron}(c)), similar to the projections shown in \cite[fig. 22]{Yamamoto:1996:CQC}.\footnote{The naming convention for the symmetries is slightly different in \cite{Yamamoto:1996:CQC}, i.e., what we refer to as a $2N-$fold symmetry is called $N-$fold in \cite{Yamamoto:1996:CQC}.} To declutter the visualization, we only display in \cref{fig:icosahedron} the lowest 7.5\% level-set of the ARP on the domain $[-6\lambda,6\lambda]^2$, with $\lambda = 2\pi$, as determined by the difference between the maximum and minimum ARP on this domain. For \cref{fig:icosahedron} we used a regular $512\times 512\times 512$ grid. Since the particles cluster about the local minima, there are many more local minima than the ones visualized in \cref{fig:icosahedron}. For an animation see \texttt{icosahedral.mp4} in the supplementary materials.

\begin{figure}
\begin{tabular}{ccc}
\includegraphics[width = 0.3\textwidth]{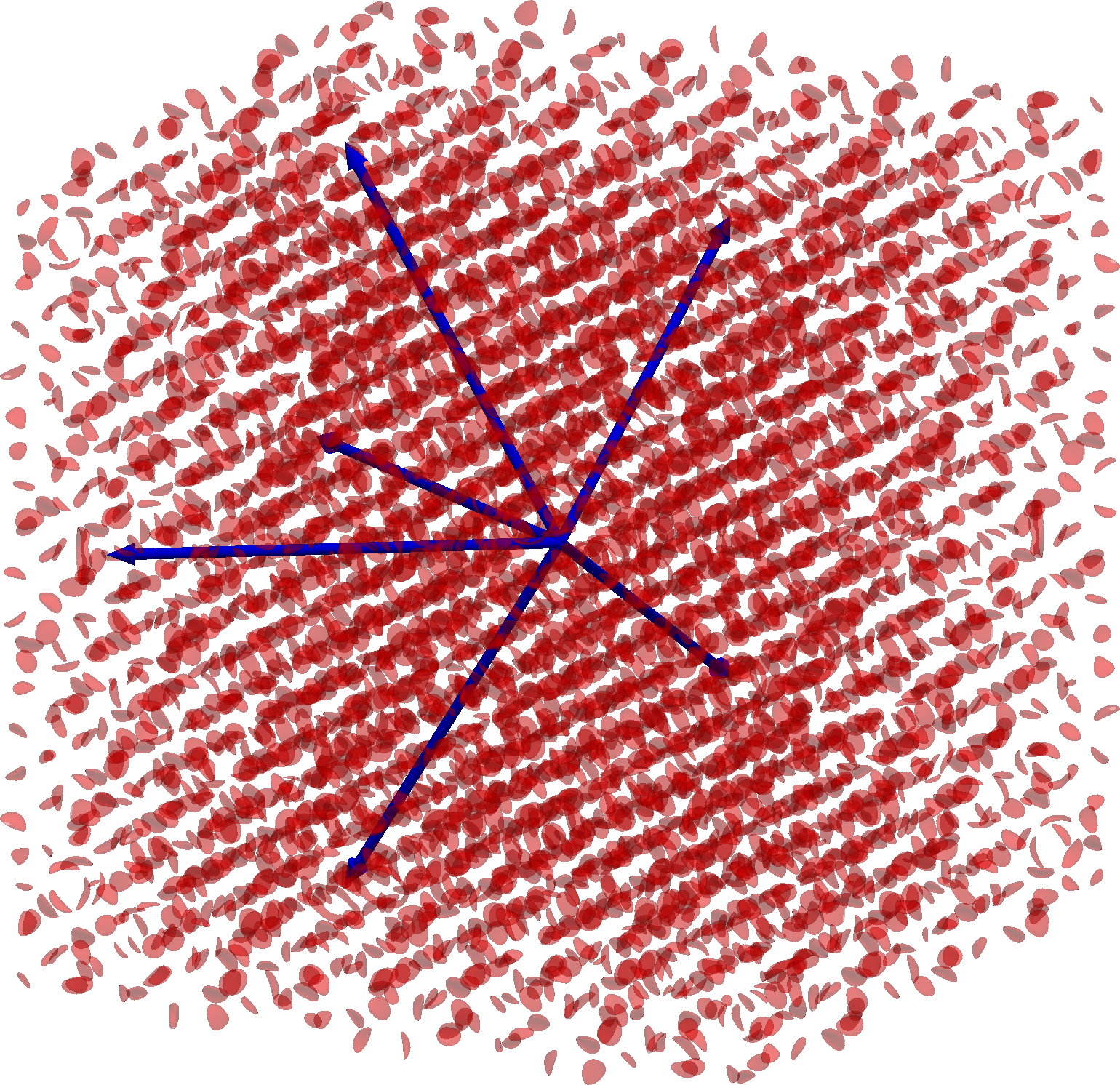} &
\includegraphics[width = 0.3\textwidth]{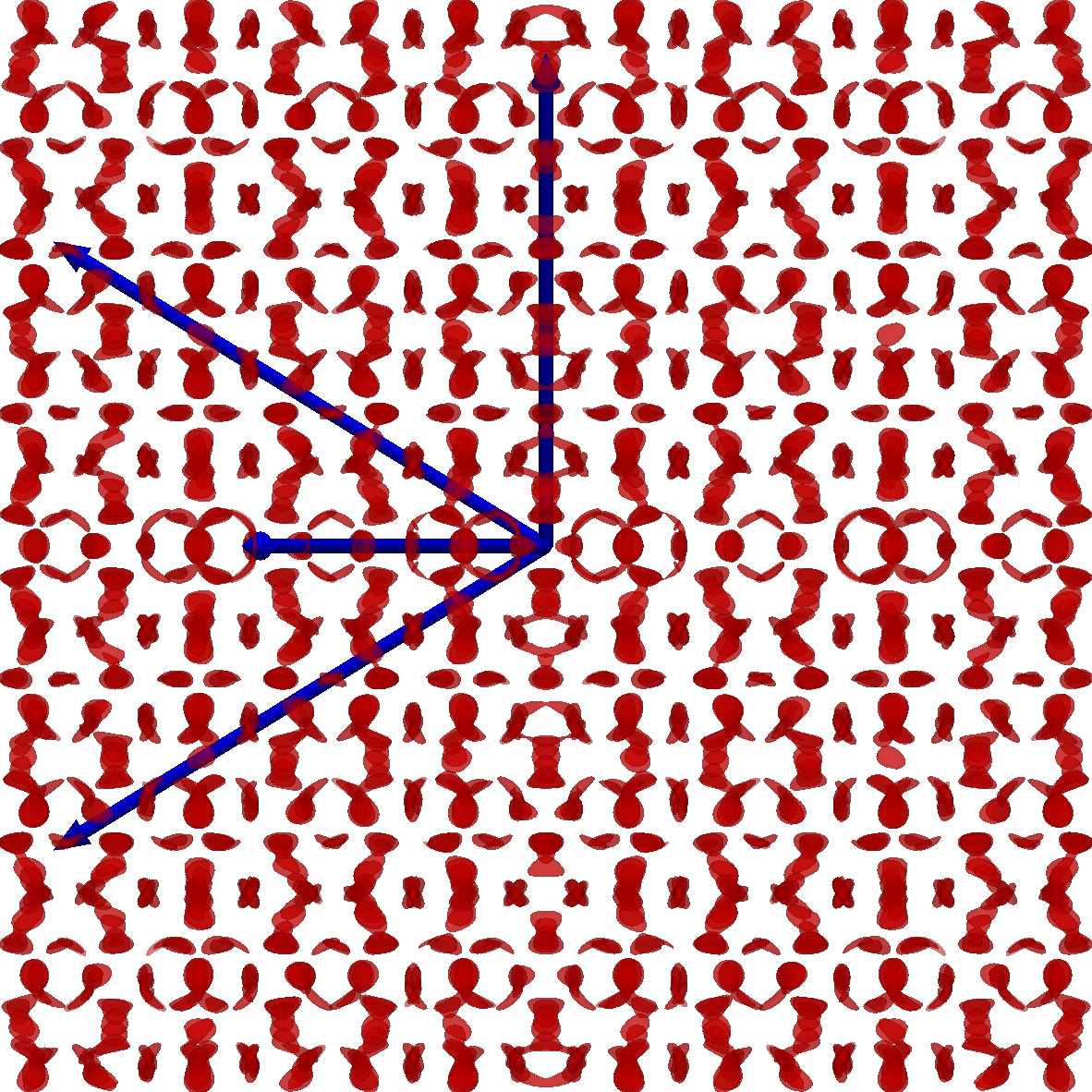} &
\includegraphics[width = 0.3\textwidth]{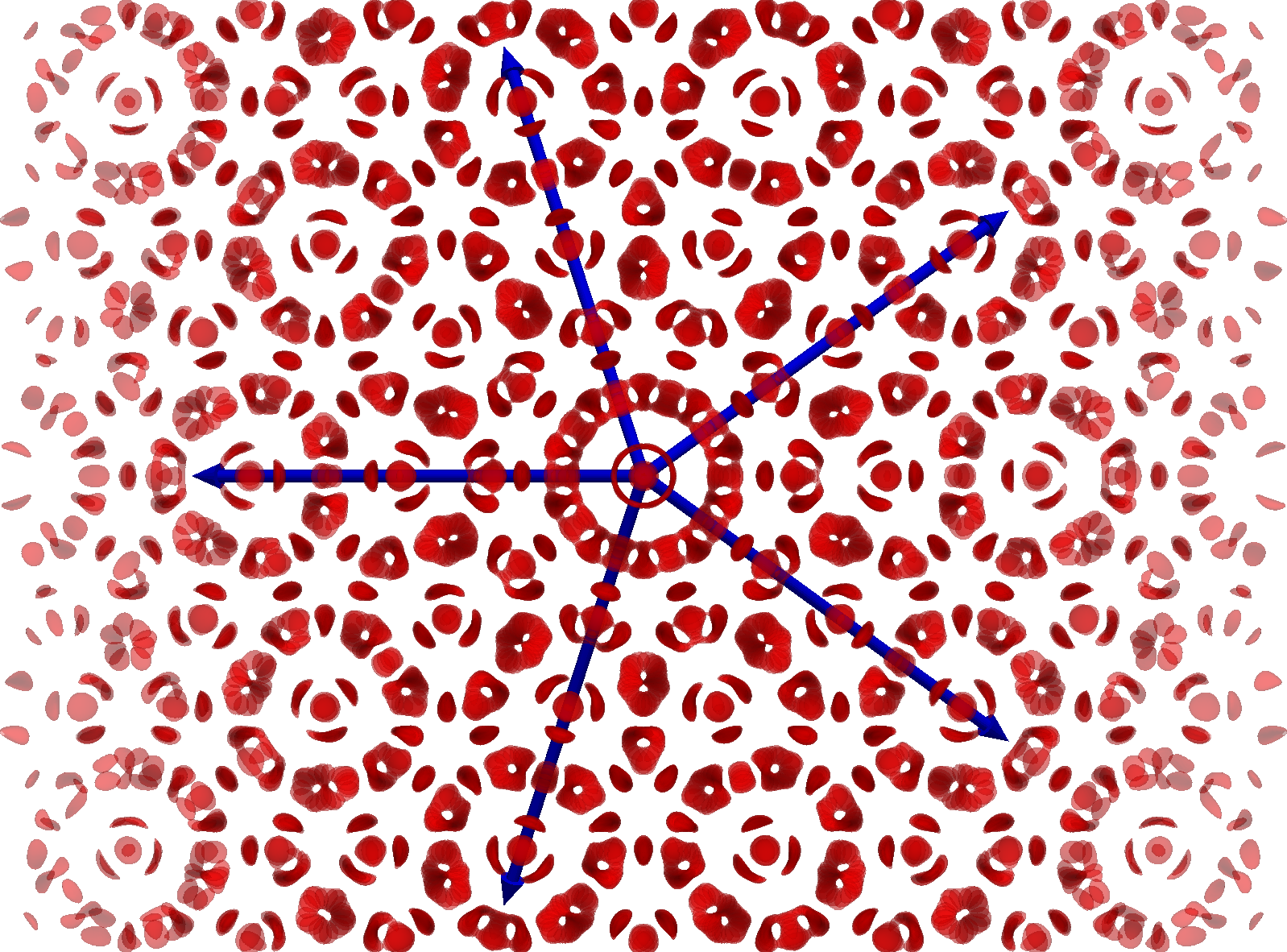} \\
(a) & (b) & (c)
\end{tabular}
\caption{Global minima of 3D ARP with icosahedral symmetry. The 3D wavevectors are the six directions that define an icosahedron, given in \eqref{eq:icosa}. ARP parameters are $\mathfrak{a}=\mathfrak{b}=1$ and transducer parameters are $\vu = [1,\ldots,1]^T\in \real^{12}$. In red, we show the lowest 7.5\% level-set of the ARP, relative to the difference between the maximum and minimum ARP on the domain  $[-6\lambda, 6\lambda]^3$, $\lambda = 2\pi$.  Blue arrows show the six transducer directions, starting at the origin. (a) View showing 3D structure. (b) Projection showing fourfold symmetry. (c) Projection showing tenfold symmetry.  See also \texttt{icosahedral.mp4} in the supplementary materials for an animation.}
\label{fig:icosahedron}
\end{figure}

\end{example}

\section{Summary and future perspectives}
\label{sec:summary}
We have shown that small particles in a fluid can be arranged in quasiperiodic
patterns. The particles cluster about the minima of an acoustic radiation
potential (ARP), which is a quadratic function of the pressure field and its
gradient. We relate the ARP to a similar higher-dimensional quantity, which is
associated with a periodic solution to the Helmholtz equation in higher
dimensions. This is related to the cut-and-project method \cite{deBruijn:1981:ATP} for describing quasicrystals. In particular, we get a good agreement between the ARP
in 2 dimensions and the tiling that corresponds to the projection used to
explain the $2-$dimensional ARP. We plan to study this agreement further, e.g., to determine whether the ARP within a tile repeats elsewhere (up to a translation, reflection, or rotation), as this would simplify the construction of ARPs.

By deriving explicit formulas for the gradient and Hessian of the ARP,
we can use the second-order sufficient optimality conditions to predict
the location of the ARP minima, which correspond to clusters of
particles. This is related to finding the minima of a
higher-dimensional ARP-like quantity with a linear constraint. The
Lagrange multiplier that is associated with this problem can be
interpreted as a change in the ARP due to certain infinitesimally small
changes in the transducer parameters. We also created a library of
possible transformations, such as translations, reflections and
rotations that can be applied to a particular ARP, and a systematic way of smoothly transitioning between two ARP configurations while maintaining a constant power. Such transitions were determined by minimizing a bound for the total ARP \eqref{eq:total:arp} in some set. Since the bound depends only on the measure of the region, the minimization of the total ARP is left for future work.

The sensitivity of the $d-$dimensional ARP to translations in $N$ dimensions is related to a diagonal unitary transformation of the transducer parameters consisting of phase changes. A natural question is to study the sensitivity to more general unitary transformations. Another possible extension of this work is to estimate the transducer parameters needed to obtain a particular ARP in dimension $d$.

Since our results rely only on a quadratic functional of the wave field, our method may be applied to other wave phenomena to create metamaterials in situations where particles can be controlled by 
waves, e.g., by light. In addition, this approach could be used to design time-dependent metamaterials as long as the timescale of transforming the structure of the metamaterial is longer than the timescale related to the mobility of the particles.
This depends in particular on other properties of the fluid and the particles that have been ignored here (e.g. the viscosity \cite{Noparast:2022:EMV}). A characterization of the possible timescales for time-dependent metamaterials is left for future work.

\appendix
\section{Details of \cref{thm:gradient} proof}
\label{app:grad}
To shorten notation, we define the linear mapping that to $\veps$ associates the diagonal matrix $\varphi(\veps) = \mD(\mK^T\veps)$.
Substituting the expansion of the diagonal matrix $\exp[i\varphi(\veps)]$ into the expression for $\mQ(\vx_0+\veps)$, we obtain
\begin{align*}
\mQ(\vx_0+\veps) &= \left[\mI + i\varphi(\veps) - \frac{1}{2}\varphi(\veps)^2 + \cdots\right]^*
\mQ(\vx_0)
\left[\mI + i\varphi(\veps) - \frac{1}{2}\varphi(\veps)^2 + \cdots \right] \\
&= \mQ(\vx_0) + [i\varphi(\veps)]^*\mQ(\vx_0) - \frac{1}{2}(\varphi(\veps)^2)^*\mQ(\vx_0) + \mQ(\vx_0)i\varphi(\veps) \\
&+ [i\varphi(\veps)]^*\mQ(\vx_0)[i\varphi(\veps)] - \frac{1}{2}\mQ(\vx_0)\varphi(\veps)^2 + \cdots
\end{align*}
where we ignored all the cubic and higher order terms in $\veps$. Using the last expansion to write the ARP we obtain
\begin{align*}
\psi(\vx_0+\veps;\vu) &= \vu^*\mQ(\vx_0+\veps)\vu \\
&= \psi(\vx_0;\vu) + \vu^*[i\varphi(\veps)]^*\mQ(\vx_0)\vu - \frac{1}{2}\vu^*([\varphi(\veps)]^2)^*\mQ(\vx_0)\vu \\
&+ \vu^*\mQ(\vx_0)[i\varphi(\veps)]\vu + \vu^*[i\varphi(\veps)]^*\mQ(\vx_0)[i\varphi(\veps)]\vu\\
&- \frac{1}{2}\vu^*\mQ(\vx_0)[\varphi(\veps)]^2\vu + \cdots.
\end{align*}
Now equate terms with the Taylor expansion
\[
\psi(\vx_0+\veps;\vu) = \psi(\vx_0;\vu) + \veps^T\nabla_{\vx}\psi(\vx_0;\vu) + \frac{1}{2}\veps^T\nabla^2_{\vx}\psi(\vx_0;\vu)\veps + \cdots.
\]
Starting with the linear terms, we have
\[
\begin{aligned}
\vu^*[i\varphi(\veps)]^*\mQ(\vx_0)\vu + \vu^*\mQ(\vx_0)i\varphi(\veps)\vu &= -i\vu^*[\varphi(\veps)]^*\mQ(\vx_0)\vu + i\vu^*\mQ(\vx_0)\varphi(\veps)\vu\\
&= -i\overline{\vu^*[\mQ(\vx_0)]^*\varphi(\veps)\vu} + 
i\vu^*\mQ(\vx_0)\varphi(\veps)\vu \\
&= -i\overline{\vu^*[\mQ(\vx_0)]^*\diag(\vu)
\begin{bmatrix}
\mK^T \\
-\mK^T
\end{bmatrix}
\veps}\\
&~~~+ i\vu^*\mQ(\vx_0)\diag(\vu)
\begin{bmatrix}
\mK^T \\
-\mK^T
\end{bmatrix}
\veps \\
&= -i\vu^T\mQ(\vx_0)^T\diag(\overline{\vu})
\begin{bmatrix}
\mK^T \\
-\mK^T
\end{bmatrix}
\veps\\
&~~~+ i\vu^*\mQ(\vx_0)\diag(\vu)
\begin{bmatrix}
\mK^T \\
-\mK^T
\end{bmatrix}
\veps.
\end{aligned}
\]
Thus, we have
\[
\begin{aligned}
\nabla_{\vx} \psi(\vx_0;\vu) &= -i[\mK, -\mK]\diag(\overline{\vu})\mQ(\vx_0)\vu + i[\mK, -\mK]\diag(\vu)\mQ(\vx_0)^T\overline{\vu} \\
&= i[\mK,-\mK]\left(\overline{\diag(\overline{\vu})\mQ(\vx_0)\vu}-\diag(\overline{\vu})\mQ(\vx_0)\vu\right) \\
&= 2[\mK,-\mK]\Im(\diag(\overline{\vu})\mQ(\vx_0)\vu).
\end{aligned}
\]
The quadratic term of the expansion gives an expression for the Hessian of $\psi$. Indeed, the quadratic terms are
\[
\vu^*[i\varphi(\veps)]^*\mQ(\vx_0)i\varphi(\veps)\vu
- \frac{1}{2}\vu^*([\varphi(\veps)]^2)^*\mQ(\vx_0)\vu
- \frac{1}{2}\vu^*\mQ(\vx_0)[\varphi(\veps)]^2\vu.
\]
The first quadratic term is
\[
\begin{aligned}
\vu^*[i\varphi(\veps)]^*\mQ(\vx_0)i\varphi(\veps)\vu 
&= -i\vu^*[\varphi(\veps)]^*\mQ(\vx_0)i\varphi(\veps)\vu \\
&= \vu^*[\varphi(\veps)]^*\mQ(\vx_0)\varphi(\veps)\vu \\
&= (\varphi(\veps)\vu)^*\mQ(\vx_0)\varphi(\veps)\vu \\
&= \left(\diag(\vu)
\begin{bmatrix}
\mK^T \\
-\mK^T
\end{bmatrix}
\veps\right)^*\mQ(\vx_0)\diag(\vu)
\begin{bmatrix}
\mK^T \\
-\mK^T
\end{bmatrix}
\veps \\
&= \veps^*
\begin{bmatrix}
\mK^T \\
-\mK^T
\end{bmatrix}^*
(\diag(\vu))^*\mQ(\vx_0)\diag(\vu)
\begin{bmatrix}
\mK^T \\
-\mK^T
\end{bmatrix}
\veps \\
&= \veps^T[\mK,-\mK]\diag(\overline{\vu})\mQ(\vx_0)\diag(\vu)
\begin{bmatrix}
\mK^T \\
-\mK^T
\end{bmatrix}
\veps.
\end{aligned}
\]
The second quadratic term gives
\[
\begin{aligned}
-\frac{1}{2}\vu^*(\varphi(\veps)\varphi(\veps))^*\mQ(\vx_0)\vu
&= -\frac{1}{2}\vu^*\mD(\mK^T\veps)\mD(\mK^T\veps)\mQ(\vx_0)\vu \\
&= -\frac{1}{2}(\mD(\mK^T\veps)\vu)^*\mD(\mK^T\veps)\mQ(\vx_0)\vu \\
&= -\frac{1}{2}\left(\diag(\vu)
\begin{bmatrix}
\mK^T \\
-\mK^T
\end{bmatrix}
\veps\right)^*
\mD(\mK^T\veps)\mQ(\vx_0)\vu \\
&= -\frac{1}{2}\veps^*
\begin{bmatrix}
\mK^T \\
-\mK^T
\end{bmatrix}^*
\diag(\overline{\vu})\mD(\mK^T\veps)\mQ(\vx_0)\vu \\
&= -\frac{1}{2}\veps^T[\mK, -\mK]\diag(\overline{\vu})\diag(\mQ(\vx_0)\vu)
\begin{bmatrix}
\mK^T \\
-\mK^T
\end{bmatrix}\veps.
\end{aligned}
\]
The third quadratic term is the conjugate transpose of the second quadratic term, so combining the three quadratic terms and setting them equal to $\frac{1}{2}\veps^T\nabla^2_{\vx}\psi(\vx_0;\vu)\veps$, we have
\[
\begin{aligned}
\nabla^2_{\vx}\psi(\vx_0;\vu) &= 2\bigg\{
[\mK,-\mK]\diag(\overline{\vu})\mQ(\vx_0)\diag(\vu)
\begin{bmatrix}
\mK^T \\
-\mK^T
\end{bmatrix}\\
&~~~-\frac{1}{2}[\mK,-\mK]\diag(\overline{\vu})\diag(\mQ(\vx_0)\vu)
\begin{bmatrix}
\mK^T \\
-\mK^T
\end{bmatrix} \\
&~~~- \frac{1}{2}[\mK,-\mK]\diag(\overline{\mQ(\vx_0)\vu})\diag(\vu)
\begin{bmatrix}
\mK^T \\
-\mK^T
\end{bmatrix}\bigg\} \\
&= 2[\mK,-\mK]\bigg[\diag(\overline{\vu})\mQ(\vx_0)\diag(\vu) 
- \frac{1}{2}\diag(\overline{\vu})\diag(\mQ(\vx_0)\vu)\\
&- \frac{1}{2}\overline{\diag(\overline{\vu})\diag(\mQ(\vx_0))}\bigg]
\begin{bmatrix}
\mK^T \\
-\mK^T
\end{bmatrix} \\
&= 2[\mK,-\mK]\bigg[\diag(\overline{\vu})\mQ(\vx_0)\diag(\vu) 
-\Re[\diag(\overline{\vu})\diag(\mQ(\vx_0)\vu)]\bigg]
\begin{bmatrix}
\mK^T \\
-\mK^T
\end{bmatrix}.
\end{aligned}
\]

\bibliographystyle{siamplain}
\bibliography{herglotz_bib}

\end{document}